\documentclass[final,3p,times]{elsarticle}
\usepackage{amssymb}
 \usepackage{amsthm,amsmath}

\newtheorem{Obs}{Remark}[section]
\newtheorem{definition}{Definition}[section]

\newtheorem{teo}{Theorem}[section]
\newtheorem{proposition}{Proposition}[section]

\newtheorem{lema}{Lemma}[section]

 \usepackage{lineno}

\begin{document}

\begin{frontmatter}


\title{Local Exact Controllability to the Trajectories of the Korteweg--de Vries--Burgers 
    Equation on a Bounded Domain with Mixed Boundary Conditions}


\author[label1]{Eduardo Cerpa}
\author[label1]{Cristhian Montoya}
\author[label2]{Bingyu Zhang}

\address[label1]{Departamento de Matem\'atica, Universidad T\'ecnica Federico Santa Mar\'ia, Casilla 110-V, Valpara\'iso, Chile.}
\address[label2]{Department of Mathematical Sciences, University of Cincinnati, Cincinnati, Ohio 45221, USA.}

\begin{abstract}
This paper studies the internal control of the Korteweg--de Vries--Burgers	(KdVB) equation on a bounded domain. The diffusion coefficient is time-dependent and the boundary conditions are mixed in the sense 
that homogeneous Dirichlet and periodic Neumann boundary conditions are considered. The exact controllability to the trajectories is proven for a linearized system by using duality and getting a new Carleman estimate.  
Then, using an inversion theorem we deduce the local exact controllability to the trajectories for the 
original KdVB equation, which is nonlinear.
\end{abstract}

\begin{keyword}
Korteweg--de Vries--Burgers equation, controllability,
	 Carleman estimates.



\end{keyword}

\end{frontmatter}

% \linenumbers

\section{\normalsize Introduction} 

\quad

	 The Korteweg--de Vries (KdV) equation appears in the nineteenth century with the works of Boussinesq \cite{boussinesq},  Korteweg and de Vries
	\cite{korteweg1895xli}, \cite{miles1981korteweg}. From a physical point of view, the KdV equation represents a model for the motion of long water waves in channels of shallow
	depth, in which two different phenomenon are presents, namely, nonlinear convection and dispersion. 
	This interaction produces a wave traveling at constant speed without losing its sharp, usually called
	\textit{soliton}.
	
	The study of the KdV equation  from a control point of view   began with the work of Russell  \cite{russell1991}
         and Zhang \cite{zhang1990} in late 1980s.     Both exact control problem and stabilization problem have been intensively studied since then.  
	For internal control of the KdV equation on a periodic domain, Russell and Zhang \cite{russell1996exact}  showed that the system is locally exactly controllable and exponentially stabilizable in the space  $H^s (\mathbb{T})$ for any $s\geq 0$. 
	Their work was improved by Laurent, Rosier and Zhang \cite{laurent2010control} who showed that the system is globally exponentially stabilizable and (large time) globally exactly controllable in  $H^s (\mathbb{T})$ for any $s\geq 0$. 
	The study of the boundary controllability for the KdV equation on a bounded domain $(0,L)$
	was started  by Rosier \cite{rosier1997exact} where he  employed  only one control input.   Using  compactness--uniqueness arguments
	and the Hilbert Uniqueness method  he first showed  surprisingly  that the linearized system around the origin is exact controllable  in the space $L^2 (0,L)$ if and only if the  length $L$ of the spatial domain does not belong
	 to a set of critical values. Then assuming the length  $L$ of the spatial domain  is not critical,   he showed the nonlinear system  is locally exactly controllable in the space $L^2 (0,L)$ by using contraction mapping principle.
	If all three boundary controls are employed, Zhang  \cite{zhang1999exact}  using a different  approach  proved that  the system is locally exactly controllable  in $H^s (0,L)$ for $s\geq 0$  without  any restrictions on the spatial domain. When the linearized system 
	is not controllable,  nevertheless, one can still  prove that the nonlinear system is locally exact  controllable  in the space $L^2 (0,L)$ by using power series expansion of the
	 solutions (see \cite{coron-crepeau, cerpa2007exact, cerpa2009boundary}). Other related results can be found in \cite{glass2008some} and \cite{rosier2004control}.
	 Concerning the internal controllability for the KdV equation on a bounded domain with homogeneous Dirichlet boundary conditions,   the most recent
	work  was done by Capistrano--Filho et al.  in 
	\cite{capistrano2015internal}, where the authors obtained  some controllability   results    using an approach  based on Carleman estimates and weighted Sobolev spaces.
	%\vskip 0,3cm
	    
	 On the other side, the Burgers equation first appeared in 1940 as a
	 simplified one--dimensional model for the Navier--Stokes system \cite{burgers1940application}. Its controllability
	 properties on bounded domains are certainly different in each case (i.e., distributed controls, boundary control, 
	 initial value  control). For instance, in
	 \cite{horsin1998controllability}, Horsin studies the exact controllability on a bounded domain for the 
	 Burgers equation by means of the return method \cite{coron2007control}. In the case of boundary
	 controllability with partial measurements, the work 
	 \cite{guerrero2007remarks} done by Imanuvilov and Guerrero shows that the exact controllability property does not hold.
	 In the context of distributed controls with Dirichlet and Neumann boundary conditions, the 
	 works by Fernandez--Cara and Guerrero \cite{fernandez2007null} and Marbach \cite{marbach2014small} addressed
	 these problems.
	%\vskip 0,3cm

	As consequence of the union of the KdV and Burgers equations arise the Korteweg--de Vries--Burgers equation 
	(KdVB equation), which in our case has homogeneous Dirichlet boundary conditions and periodic Neumann boundary conditions. More 
	precisely, we consider the following system 	
	
	\begin{equation}\label{intro.mainsyst.KdVB}
  		\left\{
   		\begin{array}{llll}
        y_t+y_{xxx}-\nu(t) y_{xx}+yy_x=F(x,t) &\mbox{in} &(0,L)\times(0,T),\\
        y(0,t)=y(L,t)=0  &\mbox{in}& (0,T),\\
         y_x(0,t)=y_x(L,t)  &\mbox{on}& (0,T),\\
        y(\cdot,0)=y_0(\cdot) &\mbox{in}& (0,L),        
    	\end{array}\right.
	\end{equation}
    where $y=y(x,t)$ represents the surface elevation of the water wave at time $(0,T)$ and space $(0,L)$,
    $\nu(t):=\nu_0+\tilde\nu(t)>0,$ with $\nu_0>0$ and $\tilde{\nu}(t)\geq 0$ is the diffusion
    coefficient, 
    $F=F(x,t)$ is an internal force and
    $y_0$ is the initial datum. The system
    (\ref{intro.mainsyst.KdVB}) can be viewed as  a model of propagation of long water waves in channels of
    shallow depth, whose  solutions depend on the nonlinearity, dispersion, and dissipation. Moreover, by introducing a variable coefficient $\nu(t)$, the KdVB equation (\ref{intro.mainsyst.KdVB}) is useful to describe cosmic plasmas phenomena \cite{gao2015variety}, 
	\cite{mamun2002role}.  Respect to the boundary conditions, they appear in order to symmetry the operator. Thus, studying the 	controllability of our system can help to build for instance some feedback laws requiring that the underlying operator is skew-adjoint.
Besides, we can explicitly mention the difficulty appearing with these boundary conditions: the hidden regularity $L^2(0,T;H^1(0,L))$ is not implied by the third order term. That is the reason that the Laplacian is added.    

    From a mathematical point of view, there exist several results for the KdVB equation in both bounded and unbounded
    domains,  concerning the global and local well--posedness problem \cite{cavalcanti2014global}, 
    \cite{li2017well}, \cite{feng2007traveling} and \cite{bona1985travelling}; the optimal control problem
    \cite{boulanger2015sparse}, \cite{chen2017bang}; the internal controllability problem on unbounded domain
	\cite{gallego2018controllability}; and
    the boundary feedback stabilization problem \cite{jia2016boundary}. As far as we know, the internal controllability problem for (\ref{intro.mainsyst.KdVB})
    has not been studied and thus, our paper will fill this gap. 
   % \vskip 0,2cm

    Throughout our work, we will use the following notation: let $\omega\subset (0,L)$ be a nonempty open subset
    and let $Q=(0,L)\times (0,T)$, for $T>0$. The main result of this paper is related to the local exact controllability
    to the trajectories of the KdVB equation
	\begin{equation}\label{intro.control.system}
    	\left\{
    	\begin{array}{llll}
        y_t+y_{xxx}-\nu(t) y_{xx}+yy_x=v1_{\omega\times(0,T)} &\mbox{in} &Q,\\
        y(0,t)=y(L,t)=0  &\mbox{on}& (0,T),\\
        y_x(0,t)=y_x(L,t)  &\mbox{on}& (0,T),\\
        y(\cdot,0)=y_0(\cdot) &\mbox{in}& (0,L),        
    	\end{array}\right.
	\end{equation}
    where $v=v(x,t)$ stands for the control, which acts in the  domain $\omega\times(0,T)$.    \vskip 0,3cm

    Let us now introduce the concept of \textit{exact controllability to the trajectories} for the 
    Korteweg--de Vries--Burgers (KdVB) 
    equation. The goal is to reach (in finite time $T$) any point on a given trajectory of the same operator. 
    Let $\overline{y}$ be a solution of the uncontrolled KdVB equation:
	\begin{equation}\label{intro.trajectory.system}
    	\left\{
    	\begin{array}{llll}
        \overline{y}_t+\overline{y}_{xxx}-\nu(t)\overline{y}_{xx}+\overline{y}  \ \overline{y}_x=0
         &\mbox{in} &Q,\\
        \overline{y}(0,t)=\overline{y}(L,t)=0  &\mbox{on}& (0,T),\\
        \overline{y}_x(0,t)=\overline{y}_x(L,t)  &\mbox{on}& (0,T),\\
        \overline{y}(\cdot,0)=\overline{y}_0(\cdot) &\mbox{in}& (0,L).        
    	\end{array}\right.
	\end{equation}
	We  look for a control $v$ such that the solution of (\ref{intro.control.system})
	satisfies:
	\begin{equation}\label{intro.identity.exactlocal}
		y(\cdot, T)=\overline{y}(\cdot,T)\quad \mbox{in}\,\ (0,L).
	\end{equation}
			
	In this paper we  will show that 
	for any given trajectory $\overline{y}$, which is a solution of (\ref{intro.trajectory.system}), there exists 
	 a  $\delta>0$ such that, for any $y_0\in X$ (an appropriate Banach space) satisfying
	\begin{equation}\label{intro.ine.smalldata}
		\|y_0-\overline{y}_0\|_X\leq\delta, 
	\end{equation}
	one can find a control $v$  such that the system  (\ref{intro.control.system})  admits a solution  $y(x,t) $ satisfying 
	(\ref{intro.identity.exactlocal}).
	
	Here  we assume 
		%\begin{equation}\label{intro.regularity.trajectory.data}
		%\overline{y}_0\in H^s(0,L)	
	%\end{equation}
	%and 
	\begin{equation}\label{intro.regularity.trajectory}
		\overline{y}\in C([0,T];H^{s}(0,L))\cap L^2(0,T;H^{s+1}(0,L))	
	\end{equation}
	 for some $s\in [0,3]$.

	To prove the exact controllability to the trajectory, we consider two relevant control systems, namely, the linearized system of
	(\ref{intro.control.system}) around $\overline{y}$ which is
	\begin{equation}\label{intro.linearcontrol.sys}
   		 \left\{
    	\begin{array}{llll}
        y_t+y_{xxx}-\nu(t) y_{xx}+\overline{y}y_x+y\overline{y}_x=f+v1_{\omega\times(0,T)} &\mbox{in} &Q,\\
        y(0,t)=y(L,t)=0  &\mbox{on}& (0,T),\\
        y_x(0,t)=y_x(L,t)  &\mbox{on}& (0,T),\\
        y(\cdot,0)=y_0(\cdot) &\mbox{in}& (0,L)        
    	\end{array}\right.
	\end{equation}	
	 and the adjoint system associated to (\ref{intro.linearcontrol.sys})
	\begin{equation}\label{intro.adjointsys}
    	\left\{
   		 \begin{array}{lll}
        -\varphi_t-\varphi_{xxx}-\nu(t)\varphi_{xx}-\overline{y}\varphi_x=g &\mbox{in} &Q,\\
        \varphi(0,t)=\varphi(L,t)=0  &\mbox{on}& (0,T),\\
        \varphi_x(0,t)=\varphi_x(L,t)  &\mbox{on}& (0,T),\\
        \varphi(\cdot,T)=\varphi_T(\cdot) &\mbox{in}& (0,L).        
    	\end{array}\right.
	\end{equation}
	
	Our strategy is as follows:
	\begin{enumerate}
	\item [i)]  Establish  first a global Carleman inequality for the system (\ref{intro.adjointsys}). More
	 	precisely, we will prove the following Theorem:
		\begin{teo}\label{intro.prop.carleman}
			Let $\nu\in L^\infty(0,T)$ and assume that $\overline{y}$ satisfies
			(\ref{intro.regularity.trajectory}). 
			Then, there exist two
			positive constants $s_0, C$ 
			depending on $L$ and $\omega$ such that, for every $\varphi_T\in L^2(0,L)$ and $g\in L^2(Q)$, the 
			corresponding solution to (\ref{intro.adjointsys}) satisfies:
			\begin{equation}\label{intro.ine.carleman1}
			\begin{aligned}
    		\displaystyle\iint\limits_{Q}[s^5\xi^5|\varphi|^2+&s^3\xi^3|\varphi_x|^2+s\xi|\varphi_{xx}|^2]
        	e^{-4s\hat\alpha}dxdt
        	\\
        	&\hspace{-2cm}\leq C\Bigl(\displaystyle\iint\limits_{Q}|g|^2e^{-2s\hat{\alpha}}dxdt
        	+s^9\displaystyle\iint\limits_{\omega\times(0,T)}\xi^{9}e^{-6s\breve{\alpha}
        	+2s\hat{\alpha}}|\varphi|^2 dxdt\Bigr),
        	\end{aligned}	
			\end{equation}
		 	for every $s\geq s_0$. 
		\end{teo}
		The estimate \eqref{intro.ine.carleman1} allows us to prove a null controllability result for the linear
		system \eqref{intro.linearcontrol.sys} with right--hand side satisfying suitable decreasing properties 
		near $t=T$. Theorem \ref{intro.prop.carleman}  will be proved using the same  approach  as in  
		\cite{guerrero2018local, baudouin2014determination, capistrano2015internal}.
	\item [ii)] Then establish  the local exact controllability to the trajectories for the KdVB equation.
	 	Here, fixed point arguments will be used  to prove Theorem \ref{main_theorem2} given below. 
		\begin{teo}\label{main_theorem2} Let $T>0$ be given, Assume 
 $\overline{y} \in C([0,T]; L^2 (0,L))\cap L^2 (0,T; H^1 (0,L))$  be the solution of  (\ref{intro.trajectory.system}).   Then there exists a $\delta >0$ such that for  $y_0\in L^2(0,L)$
 % Then, the local exact controllability to the 
 			%trajectories of \eqref{intro.control.system} satisfying 
 			%\eqref{intro.regularity.trajectory.data}--\eqref{intro.regularity.trajectory} holds. i,e.,
 			%there exists $\delta>0$  such that, for any  $\overline{y}_0$ and $y_0$
			 satisfying \eqref{intro.ine.smalldata}, 
 			one can find a function control 
 			$v\in L^2(0,T;L^2(\omega))$  such that \eqref{intro.control.system}  admits a solution $y$  satisfies 
 			\begin{equation*}
 			y(\cdot, T)=\overline{y}(\cdot,T)\quad \mbox{in}\,\ (0,L).
 			\end{equation*}    
		\end{teo}
	\end{enumerate}	 
	
	The paper is organized as follows. In Section \ref{section.wp}, we prove the local well-posedness of the system
    \eqref{intro.mainsyst.KdVB}. In Section \ref{section.carleman}, we establish a Carleman inequality for 
    the adjoint system \eqref{intro.adjointsys}, which is associated to the linearized KdVB equation. 
    In other words, we prove Theorem \ref{intro.prop.carleman}. In section
	\ref{section.null.controllability.linear}, we deal with the null controllability for a linearized system 
	with a right--hand side in $L^2(0,L)$. Finally, in Section \ref{section.local.exact.trajetories}, the proof 
	of Theorem \ref{main_theorem2} is given.
    
%%%%%%%%%%%%%%%%%%%%%%%%%%%%%%%%%%%%%%%%%%%%%%%%%%%
%%%%%%%%% 1.WELL--POSEDNESS %%%%%%%%%%%%%%%
%%%%%%%%%%%%%%%%%%%%%%%%%%%%%%%%%%%%%%%%%%%%%%%%%%%    
\section{\normalsize{Well--posedness}}\label{section.wp}
\subsection{\normalsize{Linear case}}\label{section.wellposedness.l}

	\quad
	In  this subsection we establish   the well--posedness of  the system    
	\begin{equation}\label{wp.main.linear.system}
    \left\{
    \begin{array}{llll}
        y_t+y_{xxx}-\nu(t)y_{xx}+\overline{y}y_x+\overline{y}_xy=f &\mbox{in} &Q,\\
        y(0,t)=y(L,t)=0  &\mbox{on}& (0,T),\\
        y_x(0,t)=y_x(L,t)  &\mbox{on}& (0,T),\\
        y(\cdot,0)=y_0(\cdot) &\mbox{in}& (0,L),        
    \end{array}\right.
	\end{equation} 
	where $\overline{y}$ satisfies \eqref{intro.trajectory.system}. First we consider the following 
	linear problem 
	\begin{equation}\label{wp.main.linear.system.aux}
    \left\{
    \begin{array}{llll}
        y_t+y_{xxx}-\nu_0y_{xx}=f &\mbox{in} &Q,\\
        y(0,t)=y(L,t)=0  &\mbox{on}& (0,T),\\
        y_x(0,t)=y_x(L,t)  &\mbox{on}& (0,T),\\
        y(\cdot,0)=y_0(\cdot) &\mbox{in}& (0,L),        
    \end{array}\right.
	\end{equation} 
	where $\nu_0>0 $ is a  constant.
	
    \noindent %The next result ensures that the problem \eqref{wp.main.linear.system.aux} is well posed 
    %in $L^2(0,L)$. 
	\begin{proposition}\label{wp.prop.linear_case}  Let $T>0$ be given. For any $y_0\in L^2 (0,L)$  and 
     $f\in L^1(0,T;L^2(0,L))$, \eqref{wp.main.linear.system.aux} admits  a unique mild solution 
     $y\in C([0,T];L^2(0,L))$ satisfying 
      $$\|y\|_{C([0,T];L^2(0,L))} \leq  C(\|y_0\|_{L^2(0,L)}
     +\|f\|_{L^1(0,T;L^2(0,L))})$$ 
     where $C>0$ is a constant independent of $y_0 $ and $f$.
    	
	\end{proposition}
	\begin{proof}  Consider the operator 
    	$A:=-\partial^3_{x}+\nu_0\partial_{x}^2$ defined on 
    	$$\mathcal{D}(A):=\{u\in H^3(0,L)\cap H^1_0(0,L): u(0)=u(L)=0,\,\, u_x(0)=u_x(L)\}\subset L^2(0,L).
    	$$
    	  For any 
     $\varphi\in \mathcal{D}(A)$,
     $$\langle A\varphi, \varphi\rangle_{L^2(0,L)}=-\displaystyle\int\limits_{0}^L\varphi_{xxx}\varphi\, dx
     +\nu_0\displaystyle\int\limits_{0}^L\varphi_{xx}\varphi\, dx
     =-\nu_0\displaystyle\int\limits_{0}^L|\varphi_x|^2\, dx\leq 0.$$
    Thus $A $ is dissipative.  Similarly, one can verify that $A^*$ is  also dissipative. Thus, the operator $A$ generates a 
     strongly
     semigroup $\{S(t)\}_{t\geq 0}$ of contractions in $L^2(0,L)$ by   the Lumer--Phillips 
     Theorem ( see \cite{pazy2012semigroups}, Corollary 4.4, page 15). Hence, for any $y_0\in L^2(0,L)$, $T>0$ and 
     $f\in L^1(0,T;L^2(0,L))$, \eqref{wp.main.linear.system.aux} admits  a unique mild solution 
     $y\in C([0,T];L^2(0,L))$, 
     given by the formula 
     \begin{equation}\label{wp.formula.semigroup}
         y(t)=S(t)y_0+\displaystyle\int\limits_{0}^t S(t-s)f(s)ds,\quad \forall t\in [0,T]
     \end{equation}
     and depending continuously on the data, i.e., 
     $$\|y\|_{C([0,T];L^2(0,L))}:=\sup\limits_{t\in [0,T]}\|y\|_{L^2(0,L)}\leq (\|y_0\|_{L^2(0,L)}
     +\|f\|_{L^1(0,T;L^2(0,L))}).$$ 
     This completes the proof of Proposition \ref{wp.prop.linear_case}.
     \end{proof}     
\begin{Obs}
    Observe that if the initial data $y_0$ belongs to $\mathcal{D}(A)$ and $f\in C^1([0,T];L^2(0,L))$ or 
    $f\in L^1(0,T;\mathcal{D}(A))\cap C([0,T];L^2(0,L))$, the system \eqref{wp.main.linear.system.aux} admits
    a unique classical solution, in other words, $y$ belongs to
    $$C([0,T];L^2(0,L))\cap C^1((0,T];L^2(0,L))\cap C((0,T];\mathcal{D}(A)),$$
    which can be expressed as \eqref{wp.formula.semigroup}. The reader interested can see 
    [\cite{pazy2012semigroups}, Corollary 2.2, page 106] for more details.  
\end{Obs}
 
    	The following lemma  reveals  a global Kato smoothing property of the mild solutions of 
    \eqref{wp.main.linear.system.aux}. 
   
   \begin{lema}\label{wp.lemma.katoproperty1}
    For every $T>0$,  $f\in L^1(0,T;L^2(0,L))$ and $y_0\in L^2 (0,L)$, the corresponding mild solution of 
    \eqref{wp.main.linear.system.aux} belongs to $C([0,T];L^2(0,L))\cap L^2(0,T;H^1(0,L))$ and satisfies
    \begin{equation*}\label{wp.estimate.kato1}
        \|y\|_{L^{\infty}(0,T;L^2(0,L))}
        +\|y\|_{L^2(0,T;H^1(0,L))}\leq C(\|y_0\|_{L^2(0,L)}+\|f\|_{L^1(0,T;L^2(0,L))}),
    \end{equation*}
    for some positive constant $C$ dependent of $\nu_0$. 
     Furthermore, the term $y y_x$ belongs to 
    $L^1(0,T;L^2(0,L))$ and it satisfies the estimate 
    \begin{equation*}\label{wp.estimate.kato1.2} 
        \|yy_x\|_{L^1(0,T;L^2(0,L))}\leq C\|y\|^2_{L^2(0,T;H^1(0,L))},
    \end{equation*}
    for some constant $C>0$ dependent of $\nu_0$.
\end{lema}
\begin{proof} The proof  follows the same ideas  in \cite{jia2016boundary}, 
	it is  therefore omitted here.     
	\end{proof}
       
     \noindent Now we recall  three additional Lemmas on sharp Kato smoothing property of the linear KdVB systems.

     The first  one is  for the linear KdVB equation posed on the whole line $\mathbb{R}$.
     \begin{equation}\label{app.ivp1}
    \left\{
    \begin{array}{ll}
        w_t+w_{xxx}-\nu_0 w_{xx}=0,&\quad x\in \mathbb{R},\, t\in (0,+\infty),\\
        w(x, 0)=w_0(x),&\quad x\in \mathbb{R}.
    \end{array}
    \right.
\end{equation}
   
\begin{lema}\label{app.lemma.1}
    For a   given $0\leq s\leq 3$ and $w_0\in H^s(\mathbb{R})$,  the solution of problem \eqref{app.ivp1} satisfies
    \begin{equation}\label{app.lema1.inequality}
        \sup\limits_{x\in\mathbb{R}}\Big( \|w(x,\cdot)\|_{H^{\frac{s+1}{3}}(0,+\infty)}
        +\|w_x(x,\cdot)\|_{H^{\frac{s}{3}}(0,+\infty)}\Bigr)\leq C\|w_0\|_{H^s(\mathbb{R})},
    \end{equation}
    for some positive constant $C$.
\end{lema}

The second one is  for solutions of   system  \eqref{wp.main.linear.system.aux}.  

\begin{lema}\label{wp.remark.aux1}
     For given  $y_0\in L^2(0,L)$ and $f\equiv 0$, 
    the unique solution $y$ of \eqref{wp.main.linear.system.aux}  
    belongs to $L^{\infty}(0,L;H^{\frac{1}{3}}(0,T))$  with  $y_x\in L^\infty(0,T;L^2(0,L))$ satisfying

    \begin{equation}\label{wp.katoproperty2}
        \sup\limits_{x\in [0,L]}\Bigl(\|y(x,\cdot)\|_{H^{\frac{1}{3}}(0,T)}+\|y_x(x,\cdot)\|_{L^2(0,T)}\Bigr)
        \leq C\|y_0\|_{L^2(0,L)},
    \end{equation}
    where $C$ is a positive constant.
\end{lema}

The third one is for solutions of the following linear problem
\begin{equation}\label{wp.linearproblem.datanull}
    \left\{
    \begin{array}{llll}
        y_t+y_{xxx}-\nu_0 y_{xx}=f &\mbox{in} & (0,L)\times (0,+\infty),\\
        y(0,t)=y(L,t)=0  &\mbox{on}& (0,T),\\
        y_x(0,t)=y_x(L,t)  &\mbox{on}& (0,T),\\
        y(\cdot,0)=0 &\mbox{in}& (0,L).        
    \end{array}\right.
\end{equation} 
  %  possesses the sharp Kato smoothing property, which is given in the following Lemma.
    
\begin{lema}\label{app.lema.2} 
    For any $T>0$ and $f\in L^1(0,T;L^2(0,L))$, there exists a positive constant $C$ such that 
    the solution $y(x,t)$ of \eqref{wp.linearproblem.datanull}
    satisfies
    \begin{equation*}\label{app.lema2.inequality}
         \sup\limits_{x\in [0,L]}\Bigl(\|y(x,\cdot)\|_{H^{\frac{1}{3}}(0,T)}+\|y_x(x,\cdot)\|_{L^2(0,T)}\Bigr)
        \leq C\int\limits_0^T\|f(\cdot, s)\|_{L^2(0,L)}\,ds.   
    \end{equation*}
\end{lema}

    Combining the previous results, we obtain the following Lemma for the linear system 
    \eqref{wp.main.linear.system.aux}. 
\begin{lema}\label{app.lema.3}
    For any $T>0$,\, $f\in L^1_{loc}(0,+\infty; L^2(0,L))$ and $y_0\in L^2(0,T)$, the linear
    problem \eqref{wp.main.linear.system.aux} admits a unique solution 
    $$y\in C([0,T];L^2(0,L))\cap L^2(0,T;H^1(0,L))\cap L^\infty(0,L;H^{\frac{1}{3}}(0,T))$$
    satisfying $y_x\in L^\infty(0,L;L^2(0,T))$. 
    Furthermore, there exists a constant $C$ independent of $T,\,y_0$ and $f$ such that
    \begin{equation*}\label{app.lema3.inequality}
    \begin{aligned}
        \sup\limits_{t\in[0,T]}\|y(\cdot,t)\|_{L^2(0,L)}+ &\|y\|_{L^2(0,T;H^1(0,L))} 
        +\sup\limits_{x\in [0,L]}\Bigl(\|y(x,\cdot)\|_{H^{\frac{1}{3}}(0,T)}+\|y_x(x,\cdot)\|_{L^2(0,T)}\Bigr)
        \\
        &\leq C\Bigl(\|f\|_{L^1(0,T;L^2(0,L))}+\|y_0\|_{L^2(0,L)}\Bigr). 
    \end{aligned}
    \end{equation*}   
\end{lema}
		
	\noindent In order to build the necessary regularity which will be used later on,  we introduce a weak 
	formulation of \eqref{wp.main.linear.system.aux} for $f\in L^2(0,T;H^{-1}(0,L))$.
	\begin{definition}\label{def1}
		For $(f,y_0)\in L^2(0,T;H^{-1}(0,L))\times L^2(0,L)$ a function $y\in C([0,T];L^2(0,L))$ is 
		called a weak
		solution of \eqref{wp.main.linear.system.aux} if it satisfies the following identity
		 \begin{equation}\label{weak-solution}
		 	\int\limits_{0}^T\int\limits_{0}^L yg dxdt +(y(T),\varphi_T)_{L^2(0,L)}
		 	=\int\limits_{0}^T\langle f,\varphi\rangle_{H^{-1}(0,L)\times H^1_0(0,L)} dt
		 	+(y_0,\varphi(0))_{L^2(0,L)},	
		 \end{equation}
		 for all $(g,\varphi_T)\in L^1(0,T;L^2(0,L))\times L^2(0,L)$, where $\varphi=\varphi(g,\varphi_T)$
		 is the mild solution of
		 \begin{equation}\label{dual-system}
    		\left\{
    		\begin{array}{llll}
        	-\varphi_t-\varphi_{xxx}-\nu_0\varphi_{xx}=g &\mbox{in} & (0,L)\times (0,+\infty),\\
        	\varphi(0,t)=\varphi(L,t)=0  &\mbox{on}& (0,T),\\
        	\varphi_x(0,t)=\varphi_x(L,t)  &\mbox{on}& (0,T),\\
        	\varphi(\cdot,T)=\varphi_T &\mbox{in}& (0,L).        
   			\end{array}\right.
		\end{equation} 
	\end{definition}
	In the following proposition we prove a regularity result for \eqref{wp.main.linear.system.aux} by considering
	the pair $(f,y_0)$ belongs to  $L^2(0,T;H^{s-1}(0,L))\times H^{s}(0,L))$ for  any given  $s\in [0,3]$. 
		
	\begin{proposition}\label{wp-full.linear.s}
		Let $0\leq s\leq 3$ be given. For  any  $(f,y_0)\in L^2(0,T;H^{s-1}(0,L))\times H^{s}(0,L))$,  the system \eqref{wp.main.linear.system.aux} admits 
		a unique weak solution $y\in C([0,T];H^s(0,L))\cap L^2([0,T];H^{s+1}(0,L))$  and, 
		furthermore, there exists a positive constant $C$ such that 
		 \begin{equation}
			\|y\|_{L^2(0,T;H^{s+1}(0,L))}
			\leq C\Bigl(\|f\|_{ L^2(0,T;H^{s-1}(0,L))}+\|y_0\|_{H^{s}(0,L)}\Bigr).
    	\end{equation} 
	\end{proposition}
	
	\begin{proof}
	Consider the system
	\begin{equation*} \frac{du}{dt} =A u+ f, \quad u(0) =\phi \label{1.1}\end{equation*}
	 as defined in \eqref{wp.main.linear.system.aux}. Since  $A$ is the infinitesimal generator of a semigroup $S(t)$ 
	 in the space $L^2 (0,L)$,  it follows from the standard semigroup theory that 
 	$$\phi \in L^2 (0,L), \quad f\in L^1 (0,T; L^2 (0,L)) \implies  u\in C([0,T]; L^2 (0,L))$$
 	and moreover, there exists a constant $C>0$ such that
 	
 	\begin{equation*} \label{1.2}
 		\| u\|_{C([0,T]; L^2 (0,L))}\leq C\left( \| \phi\|_{L^2 (0,L)} + \| f\|_{L^1 ((0,T); L^2 (0,L))}\right). 
 	\end{equation*}
 	In addition,
 	$$\phi \in {\cal D} (A), \quad f\in L^1 (0,T; {\cal D} (A)) \implies  u\in C([0,T]; H^3 (0,L)) $$
 	and furthermore, there exists a constant $C>0$ such that
 	\begin{equation*} \label{1.3}
 		\| u\|_{C([0,T]; H^3 (0,L))}\leq C\left ( \| \phi\|_{H^3(0,L)} + \| f\|_{L^1( (0,T); H^3 (0,L))} \right).
 	\end{equation*}
	Taking into account that 
	$$ \frac{d}{dt} \int _0^L u^2(x,t) dx + 2\nu_0\int ^L_0 u_x^2 (x,t) dx =2\int ^L_0 f(x,t) u(x,t) $$ 
	for any $t\geq 0$,  we arrive at 
	$$ \int ^L_0 u^2(x,t) dx- \int ^L_0 u^2(x,0) dx +2\nu_0\int ^t_0 \int ^L_0 u_x^2 (x,t) dxdt 
	= 2\int ^t_0 \int ^L_0 f(x,t) u(x,t)dx,$$
	which implies that 
	\begin{equation*} \label{1.4}
  		\| u\|_{L^2 (0,T; H^1 (0,L))} \leq C \left ( \|\phi\|_{L^2 (0,L)} + \| f\|_{L^2(0,T; H^{-1} (0,L))} \right).
  	\end{equation*}
	Similarly, if we let $v= Au$, then we have 
	$$
 	 \| v\|_{L^2 (0,T; H^1 (0,L))} \leq C \left ( \|A\phi\|_{L^2 (0,L)} + \|Af\|_{L^2(0,T; H^{-1} (0,L))} \right), 
 	$$
 	which yields that
	\begin{equation*} \label{1.5}
 		 \| u\|_{L^2 (0,T; H^4 (0,L))} \leq C \left ( \|\phi\|_{H^3 (0,L)} + \| f\|_{L^2(0,T; H^{2} (0,L))} \right).
  	\end{equation*}
	By interpolation arguments,
  \begin{equation*} \label{1.6}
  	\| u\|_{L^2 (0,T; H^{1+ 3\theta}  (0,L))} \leq C \left ( \|\phi\|_{H^{3\theta} (0,L)} + \| f\|_{L^2(0,T; 
  	H^{-1+3\theta} (0,L))} \right),
  \end{equation*}
  for $0\leq \theta \leq 1$, or in equivalent form
  \begin{equation*} \label{1.7}
 	 \| u\|_{L^2 (0,T; H^{1+ s}  (0,L))} \leq C \left ( \|\phi\|_{H^s (0,L)} + \| f\|_{L^2(0,T; H^{s-1} (0,L))}
   \right),
  \end{equation*}
  for $0\leq s\leq 3$. This completes the proof of Proposition \ref{wp-full.linear.s}.
  \end{proof}
	Now, we extend the previous Proposition to the linearized system \eqref{wp.main.linear.system}. For this purpose, 
	let us introduce  the space $Y^s_T$ as follows: for any $0\leq s\leq 3$ and any $T>0$, 
	$$Y^s_T:=C([0,T];H^{s}(0,L))\cap L^2([0,T];H^{s+1}(0,L)).$$
	
	\begin{lema}\label{lema.regularity.trajectory.nonlinearterm}
		For given  $0\leq s\leq 3$ and $T>0$, there exists a positive constant $C$ such that 
		\begin{equation}\label{regularity.nonlinearterm.wp}
		\|(uv)_x\|_{L^2(0,T;H^{s-1}(0,L))}\leq C\|u\|_{Y_T^s}\|v\|_{Y_T^s}	
		\end{equation}
		and 
		\begin{equation}\label{regularity.diffusionterm.wp}
			\|\tilde{\nu}v_{xx}\|_{L^2(0,T;H^{s-1}(0,L))}\leq C\|\tilde\nu\|_{L^\infty(0,T)}\|v\|_{Y^s_T}	
		\end{equation}
		holds for any $u,v\in Y_T^s$ and $\tilde\nu\in L^\infty(0,T)$.
	\end{lema}
	\begin{proof} 
			\begin{enumerate}
		\item [i)] The case $s=0$. In this case, we have 
			$$\|uv\|^2_{L^2(Q)}\leq \int\limits_0^T\|u(\cdot,t)\|^2_{L^{\infty}(0,L)}\int\limits_0^Lv^2(x,t)dxdt	
			\leq \|v\|^2_{C([0,T];L^2(0,L))}\|u\|^2_{L^2(0,T;L^\infty(0,L))}.$$
			Taking into account that $H^1(0,L)\hookrightarrow L^\infty(0,L)$, the inequality
			\eqref{regularity.nonlinearterm.wp} is proved.
			
			On the other hand, 
			$$\|\tilde\nu v_{xx}\|^2_{L^2(0,T;H^{-1}(0,L))}\leq \sup\limits_{t\in [0,T]}|\tilde{v}|^2\|v\|^2_{L^2(0,T;H^1(0,L))}.
			$$
		\item [ii)] The case $s=1$. Following the previous steps, we have
		\begin{equation*}
		\begin{aligned}
			\|(uv)_x\|^2_{L^2(Q)}&\leq 2\int\limits_0^T\Bigl(\|v(\cdot,t)\|^2_{L^{\infty}(0,L)}\|u(\cdot,t)\|
			^2_{H^1(0,L)}+\|u(\cdot,t)\|^2_{L^{\infty}(0,L)}\|v(\cdot,t)\|^2_{H^1(0,L)} \Bigr)dt\\
			&\leq C\|u\|_{Y^1_T}\|v\|_{Y^1_T}
		\end{aligned}
		\end{equation*}
			 and
			 $$\|\tilde\nu v_{xx}\|^2_{L^2(Q)}\leq \sup\limits_{t\in [0,T]}|\tilde{v}|^2\|v\|^2_{L^2(0,T;H^2(0,L))}.
			\leq C\|\tilde\nu\|_{L^\infty(0,T)}\|v\|_{Y^1_T}.$$ 
		\end{enumerate}
		Similar arguments for $s=2, 3$ as well as interpolation properties allow to complete the proof.  
	\end{proof}

  \begin{proposition}\label{wp.regularity-linearized-system}
  	Let $T>0$ and  $s\in [0,3]$  be given and  assume $\overline{y}$ satisfies  \eqref{intro.regularity.trajectory}. Then
  	for any  $y_0\in H^s(0,L)$, the linearized system \eqref{wp.main.linear.system} admits a unique solution $y\in Y^s_{T}$.	
  \end{proposition}
  
  \begin{proof}
	The proof is developed for the case $s=0$. Similar arguments allow to extend this result for 
	$0<s\leq 3$. Let us consider $R>0$ and $0<\theta\leq\min\{1, T\}$ two appropriate constants 
	to be determined. Let 
	$B_{\theta,R}:=\{v\in Y^0_\theta: \|v\|_{Y^0_\theta}\leq R\}$ and define a map 
	$\Lambda:B_{\theta,R}\rightarrow B_{\theta,R}$ by $\Lambda(v)=y$, where $y$ is the unique solution of 
	\begin{equation*}\label{aux.system}
    \left\{
    \begin{array}{llll}
        y_t+y_{xxx}-\nu_0y_{xx}=\tilde\nu(t)v_{xx}+(\overline{y}v)_{x} &\mbox{in} &Q,\\
        y(0,t)=y(L,t)=0  &\mbox{on}& (0,T),\\
        y_x(0,t)=y_x(L,t)  &\mbox{on}& (0,T),\\
        y(\cdot,0)=y_0(\cdot) &\mbox{in}& (0,L).        
    \end{array}\right.
	\end{equation*}
	Obviously,
	\begin{equation*}\label{formula.variation.linearized}
		\Lambda(v)=S(t)y_0+\int\limits_0^tS(t-\tau)[\tilde\nu v_{xx}+(\overline{y}v)_{x}](\tau)\,d\tau.
	\end{equation*}	
	From the above representation, Proposition \ref{wp-full.linear.s} and  Lemma \ref{lema.regularity.trajectory.nonlinearterm},
	there exist positive constants $C_1,C_2$ such that
	\begin{equation}\label{estimate.map.contraction1-1}
		\|\Lambda(v)\|_{Y^0_\theta}\leq C_1\|y_0\|_{L^2(0,L)}+C_2\theta^{1/2}(\|\tilde\nu\|_{L^\infty(0,T)}
		+\|\overline{y}\|_{Y^0_T})\|v\|_{Y^0_\theta}.	
	\end{equation}
	Choose $R>0$ and $T^*=\theta$ such that
	\begin{equation*}\label{choose.theta.R-case1}
		R:=m_0C_1\|y_0\|_{L^2(0,L)}\quad\mbox{and}\quad 
		C_2T^{*1/2}(\|\tilde\nu\|_{L^\infty(0,T)}+\|\overline{y}\|_{Y^0_T})\leq\frac{1}{2n_0},\quad\forall 
		m_0,n_0\geq 2.  
	\end{equation*}
	Then, by \eqref{estimate.map.contraction1-1} we have that $\|\Lambda(v)\|_{Y^0_{T^*}}\leq R$. Furthermore,
	for every $u,v\in B_{T^*,R}$,
	 \begin{equation*}
	\begin{aligned}
        \|\Lambda(v)-\Lambda(u)\|_{Y^0_{T^*}}&\leq C_2T^{*1/2}\|\tilde\nu(v_{xx}-u_{xx})+(\overline{y}(v-u))_{x}
        \|_{L^2(0,T^*;H^{-1}(0,L))}\\
        &\leq \frac{1}{n_0}\|v-u\|_{Y^0_{T^*}}.
    \end{aligned}	
	\end{equation*}
	Therefore, $\Lambda$ is a contraction mapping on $B_{T^*,R}$ and it has a unique fixed point $u\in Y^0_{T^*}$ which is 
	the solution to the linearized problem \eqref{wp-full.linear.s} in $(0,T^*)$. Finally, from \eqref{aux.system}--\eqref{choose.theta.R-case1} 
	we can observe that $T^*\in (0,T)$ is independent on $\|y_0\|_{L^2(0,L)}$,  it implies that 
	the previous arguments can be extended on intervals $(T^*,2T^*], (2T^*,3T^*],\dots, ((n-1)T^*,nT^*=T]$. Therefore, 
	the existence of a unique solution
	of \eqref{wp.main.linear.system} in  $(0,T)$ is guaranteed. This completes the proof of Proposition
	\ref{wp.regularity-linearized-system}. 
	\end{proof}
	
	\begin{Obs}
	As consequence of Proposition \ref{wp.regularity-linearized-system} and Proposition \ref{wp-full.linear.s}, 
	for any trajectory $\overline{y}\in Y_T^s$, the solution $y$ of \eqref{wp.main.linear.system} satisfies
	\begin{equation}\label{ine.regularity.linearized.system}
		\|y\|_{Y^s_{T}}
		\leq C\Bigl(\|f\|_{ L^2(0,T;H^{s-1}(0,L))}+\|y_0\|_{H^{s}(0,L)}\Bigr),
    \end{equation}
    for some positive constant $C$. 
    \end{Obs} 
  	   
%%%%%%%%%%%%%%%%%%%%%%%%%%%%%%%%%%%%%%%%%%%%%%%%%%%
%%%%%%%%% 1.2. WELL POSEDNESS--NONLINEAR CASE %%%%%%%%%%%%%%%
%%%%%%%%%%%%%%%%%%%%%%%%%%%%%%%%%%%%%%%%%%%%%%%%%%%
\subsection{\normalsize{Nonlinear case}}\label{section.wellposedness.nl}

	\quad
	In this subsection we  turn to consider  the following   nonlinear initial boundary value problem (IBVP):
	\begin{equation}\label{sys.wellposedness-nonlinear}
  		\left\{
   		\begin{array}{llll}
        y_t+y_{xxx}-\nu(t)y_{xx}+yy_x=0 &\mbox{in} &(0,L)\times(0,+\infty),\\
        y(0,t)=y(L,t)=0  &\mbox{in}& (0,T),\\
         y_x(0,t)=y_x(L,t)  &\mbox{on}& (0,T),\\
        y(\cdot,0)=y_0(\cdot) &\mbox{in}& (0,L).       
    	\end{array}\right.
	\end{equation}
		
	\begin{proposition}\label{prop.local-well-posedness.nl}
		Let  $s\in [0,3]$ and $T>0$ be given. There exists $\delta>0$ such that for any $y_0\in H^s(0,L)$ satisfying 
		$\|y_0\|_{H^s(0,L)}\leq \delta$,
		the nonlinear system \eqref{sys.wellposedness-nonlinear} admits a unique solution $y\in Y^s_{T}$.	
	\end{proposition}
	
	\begin{proof}
		The proof follows the same scheme of the linear case. In fact, let $R>0$ be  an appropriate constant 
		to be determined. Again, we consider a map $\Lambda: B_{R}\subset Y_T^s\rightarrow Y_T^s$ by 
		$\Lambda(v)=y$ where  $y$ solves 
		\begin{equation*}\label{au.system.nonlinearcase}
  		\left\{
   		\begin{array}{llll}
        y_t+y_{xxx}-\nu(t)y_{xx}=vv_x&\mbox{in} &Q,\\
        y(0,t)=y(L,t)=0  &\mbox{in}& (0,T),\\
         y_x(0,t)=y_x(L,t)  &\mbox{on}& (0,T),\\
        y(\cdot,0)=y_0(\cdot) &\mbox{in}& (0,L).       
    	\end{array}\right.
	\end{equation*}
	In this case, 
	\begin{equation*}\label{formula.variation.linearized2}
		\Lambda(v)=S(t)y_0+\int\limits_0^tS(t-\tau)(vv_{x})(\tau)\,d\tau.
	\end{equation*}	
	Using Proposition \ref{wp-full.linear.s}, Lemma \ref{lema.regularity.trajectory.nonlinearterm} and 
	\eqref{ine.regularity.linearized.system},
	there exist positive constants $C_3,C_4$ such that
	\begin{equation}\label{estimate.map.contraction1}
		\|\Lambda(v)\|_{Y^s_T}\leq C_3\|y_0\|_{H^{s}(0,L)}+C_4\|v\|^2_{Y^s_T}.	
	\end{equation}
	Consider  $R>0$ such that
	\begin{equation}\label{choose.theta.R-case2}
		R:=m_0C_3\|y_0\|_{H^{s}(0,L)}\quad\mbox{and}\quad 
		C_4R\leq\frac{1}{2n_0},\quad\forall 
		m_0,n_0\geq 2.  
	\end{equation}
	From  \eqref{choose.theta.R-case2}, it is enough to define $\delta:=(2m_0n_0C_3C_4)^{-1}$. Then, 
	by \eqref{estimate.map.contraction1} we have that $\|\Lambda(v)\|_{Y^s_{T}}\leq R$. Furthermore,
	for every $u,v\in B_{R}$,
	 \begin{equation*}
	\begin{aligned}
        \|\Lambda(v)-\Lambda(u)\|_{Y^s_{T}}&\leq C_4\|uu_x-vv_x\|_{L^2(0,T;H^{s-1}(0,L))}\\
        &\leq C_4(\|u\|_{Y^s_{T}}+\|v\|_{Y^s_{T}})\|u-v\|_{Y^s_{T}}\\
        &\leq \frac{1}{n_0}\|v-u\|_{Y^s_{T}}.
    \end{aligned}	
	\end{equation*}
	Therefore $\Lambda $ is a contraction mapping on $B_{R}$ and it has a unique fixed point $u\in Y^s_{T}$ which is
	the solution of \eqref{sys.wellposedness-nonlinear}.
	\end{proof}

%%%%%%%%%%%%%%%%%%%%%%%%%%%%%%%%%%%%%%%%%%%%%%%%%%%
%%%%%%%%% 2.NULL CONTROLLABILITY %%%%%%%%%%%%%%%
%%%%%%%%%%%%%%%%%%%%%%%%%%%%%%%%%%%%%%%%%%%%%%%%%%%
\section{\normalsize{Carleman inequality}}\label{section.carleman}

   \quad
    In this section we will prove the Carleman estimate given in Theorem \ref{intro.prop.carleman}. To do
    this, we introduce weight functions defined as follows.
      Let $\omega$ be a nonempty open subset of $(0,L)$ and $\phi$  a positive function in $[0,L]$ such that
    	 $\phi\in C^4([0,L])$ and satisfies
	\begin{equation}\label{carl.weight1.1}
    	\phi(0)=\phi(L),\quad \phi'(0)<0,\,\,\, \phi'(L)>0,
    	\quad |\phi'(0)|=|\phi'(L)|,
	\end{equation} 
	\begin{equation}\label{carl.weight1.2}
   	 \phi''< 0\,\,\,\mbox{ in }\,\,
   	 \overline{(0,L)\backslash\omega}.    
	\end{equation}     
    Then, we consider the weight functions
	\begin{equation}\label{carl.def.weights}
    \begin{aligned}
    	\alpha(x,t):=\phi(x)\xi(t),\quad \xi(t):=\displaystyle\frac{1}{t^2(T-t)^2},\\
    	\hat\alpha(t):=\max\limits_{x\in [0,L]}\alpha(x,t),\quad 
    	\breve\alpha(t):=\min\limits_{x\in [0,L]}\alpha(x,t),\quad
    	2\hat\alpha(t)<3\breve\alpha(t).
    \end{aligned}
	\end{equation}
	Assume $\omega:=(\ell_1,\ell_2)\subset (0,L)$. It is easy to verify that $\varphi$ defined as follows satisfies \eqref{carl.weight1.1} and  \eqref{carl.weight1.2}:
	\begin{equation*}\varphi(x):=
   	\left\{
    \begin{array}{lll}
        \varepsilon x^3-3\ell_1 x^2-x+C_1 &\mbox{if} &x\in [0,\ell_1],\\
        -\varepsilon x^3+(1+3\varepsilon L^2)x+C_2 &\mbox{if} &x\in [\ell_2,L],
     \end{array}\right.
   	\end{equation*}
	where $C_1=2\varepsilon L^3+L+C_2$ and $0<\varepsilon<1$ and $C_2\gg1$.

\begin{proof}\textit{Theorem \ref{intro.prop.carleman}.} For an easier comprehension, we divide the proof in 
    several steps:
    
 %%%%%%%%%%%%%%%%%%%%%%%% STEP 1 %%%%%%%%%%%%%%%%%%%%%%%%%%%%%%%%%%       
	\noindent\textit{Step $1$. Decomposition of the solution.} In this step, we decompose the solution 
	$\varphi$ of \eqref{intro.adjointsys} in order to obtain $L^2$ regularity on the right--hand side of
	\eqref{intro.adjointsys}. In other words, let us introduce $z$ and $\psi$, the solutions of the following
	 systems 
	\begin{equation}\label{sys.for.z}
    \left\{
    \begin{array}{lll}
        -z_{t}-z_{xxx}-\nu(t)z_{xx}-\overline{y}z_x=-(\rho_0)_t\varphi &\mbox{in} &Q,\\
        z(0,t)=z(L,t)=0  &\mbox{on}& (0,T),\\
        z_x(0,t)=z_x(L,t)  &\mbox{on}& (0,T),\\
        z(\cdot,T)=0 &\mbox{in}& (0,L).        
    \end{array}\right.
	\end{equation} 
	and 
	\begin{equation}\label{sys.for.psi}
    \left\{
    \begin{array}{lll}
        -\psi_t-\psi_{xxx}-\nu(t)\psi_{xx}-\overline{y}\psi_x=-\rho_0 g\ &\mbox{in} &Q,\\
        \psi(0,t)=\psi(L,t)=0  &\mbox{on}& (0,T),\\
        \psi_x(0,t)=\psi_x(L,t)  &\mbox{on}& (0,T),\\
        \psi(\cdot,T)=0 &\mbox{in}& (0,L),        
    \end{array}\right.
	\end{equation}  
    where $\rho_0(t)=e^{-s\hat\alpha}$. By uniqueness for the linear KdVB equation, we have
    \begin{equation}\label{carl.decomposition}
    	\rho_0\varphi=z+\psi.	
    \end{equation}
  	The rest of the proof consists in making a Carleman inequality for the system \eqref{sys.for.z},
  	meanwhile, for the system \eqref{sys.for.psi} we will use the regularity result 
  	\eqref{ine.regularity.linearized.system},
  	namely
  	\begin{equation}\label{car.regularityL2.psi}
  		\|\psi\|^2_{L^2(0,T;H^2(0,L))}\leq C\|\rho_0 g\|^2_{L^2(Q)}.	
  	\end{equation}
%%%%%%%%%%%%%%%%%%%%%%%% STEP 2 %%%%%%%%%%%%%%%%%%%%%%%%%%%%%%%%%%       
	\noindent\textit{Step $2$. Change of variables and decomposition of a special operator.}  In this step, 
	we consider the differential operator satisfied by a new variable $w$, which will be $z$ up to a weight
	function. More precisely, let $w=e^{-s\alpha}z$ and $G=e^{-s\alpha}(-(\rho_0)_t\varphi+\overline{y}z_x)$.
	Then, if $L$ is the 
    operator defined by $L:=\partial_t+\partial_{xxx}+\nu(t)\partial_{xx}$, the identity 
    $e^{-s\alpha}L(e^{s\alpha}w)=-G$ is equivalent to:
    $$L_1w+L_2w=F_s$$
    where 
    \begin{equation}\label{carl.proof.decompositionoperatorL}
    \begin{array}{lll}
        L_1w:=& w_t+w_{www}+3s^2(\alpha_x)^2w_x\\
        L_2w:=& 3s(\alpha_x)w_{xx}+s^3(\alpha_x)^3w+3s(\alpha_{xx})w_x\\
    \end{array}        
    \end{equation}
     and
    \begin{equation}\label{carl.proof.decompositionoperator.rhs}
        F_s=-G-R_s,
    \end{equation}
    with 
    $$R_s:= \nu(t)s\alpha_{xx}w+s\alpha_t w+s\alpha_{xxx}w+3s^2\alpha_{xx}\alpha_{xxx}w+s\alpha_x w+w_x
        -\nu(t)(2s\alpha_x w_x-w_{xx}-s^2\alpha_{x}^2w).$$
    Therefore, 
    \begin{equation}\label{carl.proof.identity.l2}
        \|L_1w\|^2_{L^2(Q)}+\|L_2w\|^2_{L^2(Q)}+2\langle L_1w, L_2w\rangle=\|G+R_s\|^2_{L^2(Q)},
    \end{equation}
    where $\langle\cdot,\cdot\rangle$ is the $L^2(Q)$ inner product. In the next step, we will estimate the terms
    that arise of the inner product $\langle L_1w, L_2w\rangle$. This will give an inequality with global terms 
    on the left--hand side, meanwhile the local terms will appear on the right--hand side. Finally, after
	returning to the principal variable $z$, the local terms will be estimate using bootstrap arguments based 
	on the smoothing of the KdVB equation.

%%%%%%%%%%%%%%%%%%%%%%%% STEP 3   %%%%%%%%%%%%%%%%%%%%%%%%%%%%%%%%%% 
      
	\noindent\textit{Step $3$. First estimates.} In this step, we develop the nine terms appearing in 
    $\langle L_1w, L_2w\rangle$. Using integration by parts, we have:
    \begin{equation*}\label{carl.I_11}
    \begin{aligned}
        I^{1,1}&:=\langle L^1_1 w, L^1_2 w\rangle=3s\displaystyle\iint\limits_{Q}\alpha_x w_t w_{xx}dxdt\\
        &=-3s\displaystyle\iint\limits_{Q}\alpha_{xx}w_tw_xdxdt+\frac{3s}{2}\displaystyle\iint\limits_{Q}
        \alpha_{xt}|w_x|^2dxdt
        +\underbrace{3s\displaystyle\int\limits_{0}^T\Bigl(\alpha_{x}w_x w_t\Big|_{x=0}^{x=L}\Bigr)dt.
        }_{A}
    \end{aligned}
    \end{equation*}
    
    \begin{equation*}\label{I_12}
         I^{1,2}:=\langle L^1_1 w, L^2_2 w\rangle=s^3\displaystyle\iint\limits_{Q}(\alpha_x)^3w_t w dxdt
         =-\frac{3s^3}{2}\displaystyle\iint\limits_{Q}(\alpha_{x})^2\alpha_{xt}|w|^2dxdt.
    \end{equation*}
    
    \begin{equation}\label{carl.I_13}
    \begin{aligned}
        I^{1,3}:=\langle L^1_1 w, L^3_2 w\rangle &=3s\displaystyle\iint\limits_{Q}\alpha_{xx}w_x w_tdxdt\\
        &=-3s\displaystyle\iint\limits_{Q}\alpha_{xx}w_xw_{xt}dxdt
        -3s\displaystyle\iint\limits_{Q}\alpha_x w_{xx}w_tdxdt+A\\
        &=\frac{3s}{2}\displaystyle\iint\limits_{Q}\alpha_{xt}|w_x|^2dxdt+A-I^{1,1}
    \end{aligned}
    \end{equation}

    \begin{equation}\label{carl.I_21}
    \begin{aligned}
         I^{2,1}:=\langle L^2_1 w, L^1_2 w\rangle &=3s\displaystyle\iint\limits_{Q}\alpha_x w_{xx}w_{xxx}dxdt\\
         &=-\frac{3s}{2}\displaystyle\iint\limits_{Q}\alpha_{xx}|w_{xx}|^2dxdt
         +\underbrace{\frac{3s}{2}\displaystyle\int\limits_{0}^T\Bigl( \alpha_x|w_{xx}|^2\big|_{x=0}^{x=L}
         \Bigr)dt.}_{B}
    \end{aligned}
    \end{equation}
    
    \begin{equation}\label{carl.I_22}
    \begin{aligned}
         I^{2,2}&:=\langle L^2_1 w, L^2_2 w\rangle=s^3\displaystyle\iint\limits_{Q}(\alpha_x)^3ww_{xxx}\\
         &=-3s^3\displaystyle\iint\limits_{Q}(\alpha_{xx}\alpha_x)ww_{xx}dxdt
         -s^3\displaystyle\iint\limits_{Q}(\alpha_x)^3w_x w_{xx}dxdt+s^3\displaystyle\int\limits_{0}^T
         \Bigl((\alpha_x)^3ww_{xx}\Big|_{x=0}^{x=L}\Bigl)dt\\
         &=3s^3\displaystyle\iint\limits_{Q}[(\alpha_x)^2\alpha_{xx}]_{x}w w_xdxdt
         +3s^3\displaystyle\iint\limits_{Q}(\alpha_x)^2\alpha_{xx}|w_x|^2dxdt\\
         &\hspace{1cm}-3s^3\displaystyle\int\limits_{0}^T\Bigl( (\alpha_x)^2\alpha_{xx}|w_x|^2\Big|_{x=0}^{x=L}
         \Bigr)dt
         +\frac{s^3}{2}\displaystyle\iint\limits_{Q}[(\alpha_x)^3]_x|w_x|^2dxdt\\
         &\hspace{1cm}-\frac{s^3}{2}\displaystyle\int\limits_{0}^T\Bigl((\alpha_x)^3|w_x|^2\Big|_{x=0}^{x=L}
         \Bigr)dt
         +s^3\displaystyle\int\limits_{0}^T\Bigl((\alpha_x)^3w w_{xx}\Big|_{x=0}^{x=L}\Bigr)dt\\
         &= \frac{9s^3}{2}\displaystyle\iint\limits_{Q}(\alpha_x)^2\alpha_{xx}|w_x|^2dxdt
         -\frac{3s^3}{2}\displaystyle\iint\limits_{Q}[(\alpha_x)^2\alpha_{xx}]_{xx}|w|^2dxdt+\tilde{C},
    \end{aligned}
    \end{equation}
    where 
    \begin{equation*}\label{carl.I_22_C}
    \begin{aligned}
         \tilde{C}&:=\frac{3s^3}{2}\displaystyle\int\limits_{0}^T\Bigl([(\alpha)^2\alpha_{xx}]_{x}|w|^2\Big|_{x=0}^{x=L}
         \Bigr)dt
         -3s^3\displaystyle\int\limits_{0}^T\Bigl((\alpha_x)^2\alpha_{xx}|w|^2\Big|_{x=0}^{x=L}\Bigr)dt
         -\frac{s^3}{2}\displaystyle\int\limits_{0}^T\Bigl((\alpha_x)^3|w_x|^2\Big|_{x=0}^{x=L}\Bigr)dt\\
         & \hspace{1cm}+s^3\displaystyle\int\limits_{0}^T\Bigl((\alpha_x)^3w w_{xx}\Big|_{x=0}^{x=L}\Bigr)dt.
    \end{aligned}
    \end{equation*}
    
    \begin{equation}\label{carl.I_23}
    \begin{aligned}
        I^{2,3}&:=\langle L^2_1 w, L^3_2 w\rangle=3s\displaystyle\iint\limits_{Q}\alpha_{xx}w_x w_{xx}dxdt\\
        &=-3s\displaystyle\iint\limits_{Q}\alpha_{xxx}w_xw_{xx}dxdt
        -3s\displaystyle\iint\limits_{Q}\alpha_{xx}|w_{xx}|^2dxdt
        +3s\displaystyle\int\limits_{0}^T\Bigl( \alpha_{xx}w_xw_{xx}\Big|_{x=0}^{x=L}\Bigl)dt\\
        &=-3s\displaystyle\iint\limits_{Q}\alpha_{xx}|w_{xx}|^2dxdt
        +\frac{3s}{2}\displaystyle\iint\limits_{Q}\alpha_{xxx}|w_x|^2dxdt+D,
    \end{aligned}
    \end{equation}
    where 
    \begin{equation*}\label{carl.I_23_D}
        D:=3s\displaystyle\int\limits_{0}^T\Bigl(\alpha_{xx}w_xw_{xx}\Big|_{x=0}^{x=L}\Bigr)dt
        -\frac{3s}{2} \displaystyle\int\limits_{0}^T\Bigl(  \alpha_{xxx}|w_x|^2\Big|_{x=0}^{x=L}\Bigr)dt.
    \end{equation*}
    
    \begin{equation}\label{carl.I_31}
    \begin{aligned}
        I^{3,1}&:=\langle L^3_1 w, L^1_2 w\rangle=9s^3\displaystyle\iint\limits_{Q}(\alpha_x)^3w_xw_{xx}dxdt\\
        &=-\frac{9s^3}{2}\displaystyle\iint\limits_{Q}[(\alpha_x)^3]_{x}|w_x|^2dxdt
        +\frac{9s^3}{2}\displaystyle\int\limits_{0}^T\Bigl( (\alpha_x)^3|w_x|^2\Big|_{x=0}^{x=L}\Bigr)dt\\
        &=-\frac{27s^3}{2}\displaystyle\iint\limits_{Q}(\alpha_x)^2\alpha_{xx}|w_x|^2dxdt
        +\underbrace{\frac{9s^3}{2}\displaystyle\int\limits_{0}^T\Bigl((\alpha_x)^3|w_x|^2\Big|_{x=0}^{x=L}
        \Bigr)dt}_{E}.
    \end{aligned}
    \end{equation}
    
    \begin{equation}\label{carl.I_32}
    \begin{aligned}
        I^{3,2}&:=\langle L^3_1 w, L^2_2 w\rangle=3s^5\displaystyle\iint\limits_{Q}(\alpha_x)^5w w_x dxdt\\
        &=-\frac{3s^5}{2}\displaystyle\iint\limits_{Q}[(\alpha_x)^5]_{x}|w|^2dxdt
        +\frac{3s^5}{2}\displaystyle\int\limits_{0}^T\Bigl((\alpha_x)^5|w|^2\Big|_{x=0}^{x=L}\Bigr)dt\\
        &=-\frac{15s^5}{2}\displaystyle\iint\limits_{Q}(\alpha_x)^4\alpha_{xx}|w|^2 dxdt
        +\underbrace{\frac{3s^5}{2}\displaystyle\int\limits_{0}^T\Bigl((\alpha_x)^5|w|^2\Big|_{x=0}^{x=L}\Bigr)dt}
        _{F}.   
    \end{aligned}
    \end{equation}
    
     \begin{equation}\label{carl.I_33}
        I^{3,3}:=\langle L^3_1 w, L^3_2 w\rangle=9s^3\displaystyle\iint\limits_{Q}(\alpha_x)^2\alpha_{xx}|w_x|^2
        dxdt.
    \end{equation}
	
	From \eqref{carl.I_22}, \eqref{carl.I_31} and \eqref{carl.I_33} we have that 
	\begin{equation*}
		I^{2,2}+I^{3,1}+I^{3,3}=
		-\frac{3s^3}{2}\displaystyle\iint\limits_{Q}[(\phi_x)^2\phi_{xx}]_{xx}\xi^3|w|^2 dxdt+\tilde{C}+E.
	\end{equation*}

	Now, taking into account the first boundary condition of \eqref{sys.for.z} and 
	 \eqref{carl.weight1.2}, the term $I^{3,2}$ can be estimated as follows:
	\begin{equation}\label{carl.estimate_w}
		Cs^5\displaystyle\iint\limits_{Q}\xi^5|w|^2 dxdt
		-Cs^5\displaystyle\iint\limits_{\omega\times(0,T)}\xi^5|w|^2dxdt
		\leq-\frac{15s^5}{2}\displaystyle\iint\limits_{Q}(\alpha_x)^4\alpha_{xx}|w|^2dxdt,
	\end{equation}
	for any $s\geq C(L,\omega, T)$.
	
	On the other hand, if $I^{2,1}_{1}$ and $I^{2,3}_1$ denote the first terms of \eqref{carl.I_21} and
	\eqref{carl.I_23}, respectively, then
	\begin{equation}\label{carl.estimate_wxx}
		I^{2,1}_{1}+I^{2,3}_1=-\frac{9s}{2}\displaystyle\int\limits_{Q}\phi_{xx}\xi|w_{xx}|^2dxdt
		\geq Cs\displaystyle\iint\limits_{Q}\xi|w_{xx}|^2dxdt
		-Cs\displaystyle\iint\limits_{\omega\times(0,T)}\xi|w_{xx}|^2dxdt.
	\end{equation}
	for any $s\geq C(L,\omega, T)$.
	
	Now, putting together the first term of $I^{2,2}$ (denoted by $I^{2,2}_{1}$) as well as the first term of
	$I^{3,1}$ (which is denoted by $I^{3,1}_{1}$) and $I^{3,3}$, we get 
	$$I^{2,2}_{1}+I^{3,1}_{1}+I^{3,3}=0.$$
	However, from \eqref{carl.estimate_w} and \eqref{carl.estimate_wxx} we also have (after integrating by parts
	and using Young's inequality) that
	\begin{equation}\label{carl.estimate_wx}
		s^3\displaystyle\iint\limits_{Q}\xi^3|w_{x}|^2dxdt
		\leq \displaystyle\iint\limits_{Q}(s^5\xi^5|w|^2 +s\xi|w_{xx}|^2)dxdt. 
	\end{equation}
	Thus, the first term of \eqref{carl.I_13} as well as the second term of \eqref{carl.I_23} can be estimated by
	the left--hand side of \eqref{carl.estimate_wx}.
	
	Then, putting together all the computations, we get the following inequality
	\begin{equation}\label{carl.fullestimate.1}
	\begin{aligned}
        &\displaystyle\iint\limits_{Q}[s^5\xi^5|w|^2+ s^3\xi^3|w_x|^2+s\xi|w_{xx}|^2]dxdt+A+B+\tilde{C}+D+E\\
        &\leq C\Bigl(\displaystyle\iint\limits_{\omega\times(0,T)}[(s\xi)^5|w|^2+s\xi|w_{xx}|^2]dxdt
        +\|G\|^2_{L^2(Q)}+\|R_s\|^2_{L^2(Q)}\Bigr),
    \end{aligned}	
	\end{equation}
	for any $s\geq C(L,\omega, T)$.

	Observe that the last term on the right--hand side \eqref{carl.fullestimate.1} can be absorbed by the 
	left--hand side for $s\geq C(L,\omega, T,\|\nu\|_{L^\infty(0,T)})$. Furthermore, taking into account that 
	$G=[-(\rho_0)_t\varphi+\overline{y}z_x]e^{-s\alpha},$
	we can estimate the term  
	$\overline{y}z_x$ by considering the identity 
	$w_x+s\alpha_{x}w=e^{-s\alpha}z_x$ and the inequality
	\begin{equation*}
		|\overline{y}z_x|^2e^{-2s\alpha}\leq Cs|\overline{y}w_x|^2
		+Cs^2(\alpha_{x})^2|\overline{y}w|^2.
		%\leq C_3\|\overline{y}\|_{\infty}(s\xi|w_x|^2+s^2\xi^2|w|^2).
	\end{equation*}
	From \eqref{carl.decomposition},\,\eqref{car.regularityL2.psi} and the estimate $|(\rho_0)_t\varphi|\leq Cs\xi^{3/2}|\rho_0\varphi|$,  
 	we readily have that there exists a positive constant 
	$C=C(L,\omega,T,\|\nu\|_{L^\infty(0,T)}, \|\overline{y}\|_{C(0,T;L^2(0,L))\cap L^2(0,T;H^1(0,L))})$ such that
	\begin{equation*}\label{carl.fullestimate.2}
	\begin{aligned}
        &\displaystyle\iint\limits_{Q}[s^5\xi^5|w|^2+ s^3\xi^3|w_x|^2+s\xi|w_{xx}|^2]dxdt+A+B+\tilde{C}+D+E\\
        &\leq C\Bigl(\displaystyle\iint\limits_{Q}|g|^2e^{-2s\hat\alpha}dxdt+
        \displaystyle\iint\limits_{\omega\times(0,T)}[(s\xi)^5|w|^2+s\xi|w_{xx}|^2]dxdt\Bigr),
    \end{aligned}	
	\end{equation*}
	for any $s\geq C$.
	
	Finally, using the weight functions defined in \eqref{carl.weight1.1} and \eqref{carl.weight1.2} we have the
	following estimates:
	
	$A=3s\displaystyle\int\limits_{0}^T[\alpha_{x}(L,t)w_x(L,t)w_t(L,t)-\alpha_{x}(0,t)w_x(0,t)w_t(0,t)]dt
	=0.$	
	
	$$B=\frac{3s}{2}\displaystyle\int\limits_{0}^T[\alpha_x(L,t)|w_{xx}(L,t)|^2-\alpha_{x}(0,t)|w_{xx}(0,t)|^2]dt
	\geq
	Cs\displaystyle\int\limits_{0}^T\xi(|w_{xx}(0,t)|^2+|w_{xx}(L,t)|^2)dt.$$
	
	\begin{equation*}
	\tilde{C}+E=4\displaystyle\int\limits_{0}^T[(\alpha_x(L,t))^3|w_{x}(L,t)|^2
	-(\alpha_{x}(0,t))^3|w_{x}(0,t)|^2]dt
	\geq Cs^3\displaystyle\int\limits_{0}^T\xi^3|w_{x}(L,t)|^2dt.
	\end{equation*}
	and
	\begin{equation*}
	\begin{aligned}
		D&=3s\displaystyle\int\limits_{0}^T[\alpha_{xx}(L,t)w_x(L,t)w_{xx}(L,t)
		-\alpha_{xx}(0,t)w_x(0,t)w_{xx}(0,t)]dt\\
        &\hspace{1cm}
        -\frac{3s}{2} \displaystyle\int\limits_{0}^T[\alpha_{xxx}(L,t)|w_x(L,t)|^2
        -\alpha_{xxx}(0,t)|w_x(0,t)|^2]dt\\
        &\leq Cs^2\displaystyle\int\limits_{0}^T\xi|w_x(L,t)|^2dt
        +C\displaystyle\int\limits_{0}^T\xi(|w_{xx}(0,t)|^2+|w_{xx}(L,t)|^2)dt.
     \end{aligned}
	\end{equation*}
	Therefore, at this moment we have the following inequality
	\begin{equation}\label{carl.fullestimate.3}
	\begin{aligned}
        \displaystyle\iint\limits_{Q}[s^5\xi^5|w|^2+&s^3\xi^3|w_x|^2+s\xi|w_{xx}|^2]dxdt
        +s^3\displaystyle\int\limits_{0}^T\xi^3|w_{x}(L,t)|^2dt\\
        &+s\displaystyle\int\limits_{0}^T\xi(|w_{xx}(0,t)|^2+|w_{xx}(L,t)|^2)dt\\
        &\hspace{-2cm}\leq  C\Bigl(\displaystyle\iint\limits_{Q}|g|^2e^{-2s\hat\alpha}dxdt+
        \displaystyle\iint\limits_{\omega\times(0,T)}[(s\xi)^5|w|^2+s\xi|w_{xx}|^2]dxdt\Bigr),
    \end{aligned}	
	\end{equation}
	for any $s\geq C$.
	
%%%%%%%%%%%%%%%%%%%%%%%% STEP 4   %%%%%%%%%%%%%%%%%%%%%%%%%%%%%%%%%%
 
\noindent\textit{Step $4$. Local estimates.} In this step, we turn back to our original function and use 
	bootstrap arguments as in \cite{capistrano2015internal} and \cite{guerrero2018local} to
	estimate the local term associated to $|w_{xx}|$.
	
	Recall that  $z=e^{s\alpha}w$. Then, a direct computation allow to obtain
	\begin{equation}
	|z_x|^2e^{-2s\alpha}\leq C(s^2\xi^2|w|^2+|w_x|^2)
	\end{equation}
 	and 
	\begin{equation}
	|z_{xx}|^2e^{-2s\alpha}\leq C(s^4\xi^4|w|^2+s^2\xi^2|w_x|^2+|w_{xx}|^2).
	\end{equation}
	On the other hand,
	\begin{equation}\label{aux.estimate.for.wxx}
	|w_{xx}|^2\leq Ce^{-2s\alpha}(s^4\xi^4|z|^2+s^2\xi^2|z_x|^2+|z_{xx}|^2)).
	\end{equation}	 
	From \eqref{aux.estimate.for.wxx} the local term given in \eqref{carl.fullestimate.3} can be written by
	\begin{equation}\label{carl.localterm}
	\displaystyle\iint\limits_{\omega\times(0,T)}[(s\xi)^5|z|^2+s^3\xi^3|z_x|^2+s\xi|
	z_{xx}|^2]
	e^{-2s\alpha}dxdt.
	\end{equation}		
	In addition, the weight functions 
	$\hat\alpha,\breve{\alpha}, \xi$ and  \eqref{carl.fullestimate.3}--\eqref{carl.localterm} 
	allow us to deduce the following inequality
	\begin{equation}\label{carl.localterm1}
	\begin{aligned}
        \displaystyle\iint\limits_{Q}[s^5\xi^5|z|^2+&s^3\xi^3|z_x|^2+s\xi|z_{xx}|^2]
        e^{-2s\hat\alpha}dxdt
        \\
        &\hspace{-2cm}\leq  C\Bigl(\displaystyle\iint\limits_{Q}|g|^2e^{-2s\hat{\alpha}}dxdt+
        \displaystyle\iint\limits_{\omega\times(0,T)}
        [s^5\xi^5|z|^2+s^3\xi^3|z_x|^2+s\xi|z_{xx}|^2]e^{-2s\breve\alpha}dxdt\Bigr),
    \end{aligned}	
	\end{equation}
	for any $s\geq C$.
	
	Using that $H^1(\omega)=(H^3(\omega),L^2(\omega))_{2/3,2}$ and
	 $H^2(\omega)=(H^3(\omega),L^2(\omega))_{1/3,2}$, the last two terms in the right--hand side of 
	\eqref{carl.localterm1} can be upper bounded as follows: 
	\begin{equation*}
		s^3\displaystyle\iint\limits_{\omega\times(0,T)}\xi^3|z_x|^2 dxdt
		\leq \underbrace{s^3\displaystyle\int\limits_0^T
		\xi^3e^{-2s\breve{\alpha}}\|z\|^{4/3}_{L^2(\omega)}\|z\|^{2/3}_{H^3(\omega)} dt}_{J_1}
	\end{equation*}
	and 
	\begin{equation*}
	s\displaystyle\iint\limits_{\omega\times(0,T)}\xi|z_{xx}|^2 dxdt
	 \leq \underbrace{s\displaystyle\int\limits_0^T
	 \xi e^{-2s\breve{\alpha}}\|z\|^{2/3}_{L^2(\omega)}\|z\|^{4/3}_{H^3(\omega)}dt.}_{J_2} 
	\end{equation*}
	Now, applying Young's inequality
	\begin{equation*}
		J_1\leq C(\varepsilon)s^{11/2}\displaystyle\int\limits_0^T\xi^{11/2}e^{-3s\breve{\alpha}
		+s\hat{\alpha}}
		\|z\|^2_{L^2(\omega)}dt 
		+\varepsilon s^{-2}\displaystyle\int\limits_0^T\xi^{-2}e^{-2s\hat{\alpha}}
		\|z\|^2_{H^3(\omega)}dt
	\end{equation*}	   
	and 
	\begin{equation*}
		J_2\leq C(\varepsilon)s^{9}\displaystyle\int\limits_0^T\xi^{9}e^{-6s\breve{\alpha}+4s\hat{\alpha}}
		\|z\|^2_{L^2(\omega)}dt 
		+\varepsilon s^{-3}\displaystyle\int\limits_0^T\xi^{-3}e^{-2s\hat{\alpha}}
		\|z\|^2_{H^3(\omega)}dt,
	\end{equation*}
	for any $\varepsilon>0$.
	
	Putting together \eqref{carl.localterm1} and the previous estimates, we have  
	\begin{equation}\label{carl.localterm2}
	\begin{aligned}
        \displaystyle\iint\limits_{Q}[s^5\xi^5|z|^2+&s^3\xi^3|z_x|^2+s\xi|z_{xx}|^2]
        e^{-2s\hat\alpha}dxdt
        \\
        &\hspace{-2cm}\leq  C\displaystyle\iint\limits_{Q}|g|^2e^{-2s\hat{\alpha}}dxdt+
        Cs^9\displaystyle\iint\limits_{\omega\times(0,T)}\xi^{9}e^{-6s\breve{\alpha}
        +4s\hat{\alpha}}|z|^2dxdt\\
		 &+\varepsilon\Biggl( s^{-2}\displaystyle\int\limits_0^T\xi^{-2}e^{-2s\hat{\alpha}}
		\|z\|^2_{H^3(\omega)}dt+s^{-3}\displaystyle\int\limits_0^T\xi^{-3}e^{-2s\hat{\alpha}}
		\|z\|^2_{H^3(\omega)}dt\Biggr),
    \end{aligned}	
	\end{equation}
	for any $s\geq C$.
	
	Finally, in order to estimate the associated terms to $\|z\|^2_{H^3(\omega)}$, we will use a
	bootstrap argument based on the smoothing effect of the KdVB equation.
	\noindent
	Let us star by defining $\tilde z:=\tilde\rho(t)z$ with 
	$\tilde\rho(t):=s^{1/2}\xi e^{-s\hat{\alpha}}$. From \eqref{sys.for.z}, we see that
	$\tilde z$ is the solution of the system
	\begin{equation}\label{carl.adjointsystem}
    \left\{
    \begin{array}{lll}
        -\tilde{z}_t-\tilde{z}_{xxx}-\nu(t)\tilde{z}_{xx}-\overline{y}\tilde{z}_x
        =\tilde\rho(\rho_0)_t\varphi -\tilde\rho_t z 
        &\mbox{in} &Q,\\
        \tilde{z}(0,t)=\tilde{z}(L,t)=0  &\mbox{on}& (0,T),\\
        \tilde{z}_x(0,t)=\tilde{z}_x(L,t)  &\mbox{on}& (0,T),\\
        \tilde{z}(\cdot,T)=0 &\mbox{in}& (0,L).        
    \end{array}\right.
	\end{equation}
	Taking into account the estimates $|\tilde\rho_t|\leq Cs^{3/2}\xi^{5/2}e^{-s\hat\alpha}, \,\, 
	|(\rho_0)_t|\leq Cs\xi^{3/2}e^{-s\hat\alpha}$, and the 
	regularity result \eqref{ine.regularity.linearized.system}, we can deduce that
	\begin{equation}\label{carl.local.regularity.l2}
		\|\tilde{z}\|^2_{L^2(0,T;H^2(\Omega))}
		\leq C\Bigl(\|s^{3/2}\xi^{5/2}e^{-s\hat\alpha}z\|^2_{L^2(Q)}+\|s^{3/2}\xi^{5/2}e^{-2s\hat\alpha}
		\varphi\|^2_{L^2(Q)}\Bigr).
	\end{equation}   
	The fact that $s^{3/2}\xi^{5/2}e^{-s\hat\alpha}$ is bounded allows us 
	to use \eqref{car.regularityL2.psi} and conclude that $\|\tilde z\|^2_{L^2(0,T;H^2(\Omega))}$ is bounded by 
	the left--hand side of 
	\eqref{carl.localterm2} and $\|\rho_0 g\|^2_{L^2(Q)}$.
	\skip 0,3cm
	Now, we define
	\begin{equation*}
		\hat z:=\hat\rho(t) z\quad\mbox{with}\quad \hat\rho(t):=s^{-1/2}\xi^{-1/2}e^{-s\hat\alpha}.
	\end{equation*}
	It is easy to see that $\hat z$ is the solution of \eqref{carl.adjointsystem} with 
	$\tilde\rho$ replaced by $\hat\rho$. Besides, from \eqref{ine.regularity.linearized.system} 
	we get 
	\begin{equation}\label{carl.localterm.regularity.h1}
		\|\hat z\|^2_{L^2(0,T;H^3(\Omega))}
		\leq C\Bigl(\|s^{1/2}\xi e^{-s\hat\alpha} z\|^2_{L^2((0,T);H^1(\Omega))}
		+\|s^{1/2}\xi e^{-2s\hat\alpha}\varphi\|^2_{L^2(0,T;H^1(\Omega))}\Bigr).		
	\end{equation}
 	Arguing as before, $\|\hat z\|^2_{L^2(0,T;H^3(\Omega))}$ is bounded by the 
 	left--hand side of \eqref{carl.localterm2} and $\|\rho_0 g\|^2_{L^2(Q)}$.\\
 	By combining \eqref{carl.localterm2}, \eqref{carl.local.regularity.l2} and 
 	\eqref{carl.localterm.regularity.h1}, we obtain in particular
 	\begin{equation}\label{carl.localterm3}
 	\begin{aligned}
    	\displaystyle\iint\limits_{Q}[s^5\xi^5|z|^2+&s^3\xi^3|z_x|^2+s\xi|z_{xx}|^2]
        e^{-2s\hat\alpha}dxdt
        +\|s^{-1/2}\xi^{-1/2}e^{-s\hat\alpha}z\|^2_{L^2(0,T;H^3(\Omega))}
        \\
        &\hspace{-2cm}\leq  C\Bigl(\displaystyle\iint\limits_{Q}|g|^2e^{-2s\hat{\alpha}}dxdt+
        s^9\displaystyle\iint\limits_{\omega\times(0,T)}\xi^{9}e^{-6s\breve{\alpha}
        +4s\hat{\alpha}}|z|^2 dxdt\Bigr)\\
		&\hspace{-1cm}+\varepsilon\Biggl( s^{-2}\displaystyle\int\limits_0^T\xi^{-2}e^{-2s\hat{\alpha}}
		\|z\|^2_{H^3(\omega)}dt+s^{-3}\displaystyle\int\limits_0^T\xi^{-3}e^{-2s\hat{\alpha}}
		\|z\|^2_{H^3(\omega)}dt\Biggr),
    \end{aligned}	
 	\end{equation}
	for any $\varepsilon>0$.
	\skip 0,2cm \noindent
	For $\varepsilon$ small enough, the last two terms in the right-hand side of \eqref{carl.localterm3}
	can be absorbed by the left--hand side. By returning  to the variable $\varphi$  the proof of 
	Theorem \ref{intro.prop.carleman} is ended. 
\end{proof}	
%%%%%%%%%%%%%%%%%%%%%%%%%%%%%%%%%%%%%%%%%%%%%%%%%%%
%%%%%%%%% CONTROLLABILITY ---LINEAR CASE %%%%%%%%%%%%%%%
%%%%%%%%%%%%%%%%%%%%%%%%%%%%%%%%%%%%%%%%%%%%%%%%%%%	
\section{\normalsize{Null controllability of the linearized system}}\label{section.null.controllability.linear}

	\quad
	In this section we will prove the null controllability for the system \eqref{intro.linearcontrol.sys}
	with a right--hand side which decays exponentially to zero when $t$ goes to $T$ \cite{fursikovimanuvilov}. In other words, we would
	like to find $v\in L^2(0,T;L^2(\Omega))$ such that the solution of    
	\begin{equation}\label{controllinear.mainsyst}
    \left\{
    \begin{array}{lll}
        y_t+y_{xxx}-\nu(t)y_{xx}+\overline{y}y_x+y\overline{y}_x=h+v1_{\omega\times(0,T)} &\mbox{in} &Q,\\
        y(0,t)=y(L,t)=0  &\mbox{on}& (0,T),\\
        y_x(0,t)=y_x(L,t)  &\mbox{on}& (0,T),\\
        y(\cdot,0)=y_0(\cdot) &\mbox{in}& (0,L),        
    \end{array}\right.
	\end{equation} satisfies
	\begin{equation}
		y(\cdot,T)=0\quad \mbox{in}\,\,(0,L)	,
	\end{equation}
 	where the function $h$ is in an appropriate weighted space. Before proving this results, we establish
 	a Carleman inequality with weight functions not vanishing in $t=0$. To do this, let 
	$\ell(t)\in C^1([0,T])$ be a positive
    function in $[0,T)$ such that $\ell(t)=T^2/4$ for all $t\in [0,T/4]$ and $\ell(t)=t(T-t)$ for all 
    $t\in [T/2,T]$. We introduce the following weight functions:	
	\begin{equation}\label{linearcontrol.carleman.weights}
    \begin{aligned}
    &\beta(x,t) = \phi(x)\tau(t), 
    \quad \tau(t)=\dfrac{1}{\ell^{2}(t)},\\
    &\widehat\beta(t) = \max_{x\in[0,L]} \beta(x,t),\quad 
    \quad \,\,\,\,\breve\beta(t) = \min_{x\in [0,L]} \beta(x,t).
    \end{aligned}
	\end{equation}
	\begin{lema}\label{lema.weight.carleman2}
	There exist positive constants $s, C$ with $C$ depending on $s,\|\nu\|_{L^\infty(0,T)},\omega,T$ such that every
	solution of \eqref{intro.adjointsys} verifies
	\begin{equation}\label{nullcontrol.ine.carleman}
	\begin{aligned}
        \displaystyle\iint\limits_{Q}[\tau^5|\varphi|^2+&\tau^3|\varphi_x|^2+\tau|\varphi_{xx}|^2]
        e^{-4s\hat\beta}dxdt+\|\varphi(0)\|^2_{L^2(0,L)}
        \\
        &\hspace{-2cm}\leq  C\Bigl(\displaystyle\iint\limits_{Q}|g|^2e^{-2s\hat{\beta}}dxdt+
        \displaystyle\iint\limits_{\omega\times(0,T)}\tau^{9}e^{-6s\breve{\beta}+2s\hat{\beta}}
		|\varphi|^2dxdt\Bigr).
	\end{aligned}	
	\end{equation}
	\end{lema}
	\begin{proof}
	By construction $\alpha=\beta$ and $\tau=\xi$ in $[0,L]\times(T/2,T)$, so that
	\begin{equation*}
    \begin{array}{ll}
    	\displaystyle\int\limits_{T/2}^T\displaystyle\int\limits_{0}^L
    	[\xi^5|\varphi|^2+\xi^3|\varphi_x|^2+\xi|\varphi_{xx}|^2]
        e^{-4s\hat\alpha}dxdt
   		 =\displaystyle\int\limits_{T/2}^T\displaystyle\int\limits_0^L
    	[\tau^5|\varphi|^2+\tau^3|\varphi_x|^2+\tau|\varphi_{xx}|^2]e^{-4s\hat\beta}dxdt.	
    \end{array}
	\end{equation*}
 	As consequence of Theorem \ref{intro.prop.carleman} we have the estimate
 	\begin{equation*}
    \begin{array}{lll}
        &\displaystyle\int\limits_{T/2}^T\displaystyle\int\limits_0^L
    	[\tau^5|\varphi|^2+\tau^3|\varphi_x|^2+\tau|\varphi_{xx}|^2]e^{-4s\hat\beta}dxdt\\
        &\leq  C\Bigl(\displaystyle\iint\limits_{Q}|g|^2e^{-2s\hat{\alpha}}dxdt+
        \displaystyle\iint\limits_{\omega\times(0,T)}\xi^{9}e^{-6s\breve{\alpha}+2s\hat{\alpha}}
		|\varphi|^2dxdt\Bigr). 
    \end{array}
    \end{equation*} 
	Next, using that $\ell(t)=t(T-t)$ for any $t\in [T/2,T]$ and
	\begin{equation*}
		e^{-2s\hat\beta}\geq C\quad\mbox{and}\quad\tau^{9}e^{-6s\breve{\beta}+2s\hat{\beta}}\geq C\,\,\mbox{in}
		\,\,\, [0,T/2], 
	\end{equation*}
	we readily have
	\begin{equation}\label{carleman2.aux1}
    	\begin{array}{lll}
       	 	&\displaystyle\int\limits_{T/2}^T\displaystyle\int\limits_0^L
    		[\tau^5|\varphi|^2+\tau^3|\varphi_x|^2+\tau|\varphi_{xx}|^2]e^{-2s\hat\beta}dxdt\\
        	&\leq  C\Bigl(\displaystyle\iint\limits_{Q}|g|^2e^{-2s\hat{\beta}}dxdt+
        	\displaystyle\iint\limits_{\omega\times(0,T)}\tau^{9}e^{-6s\breve{\beta}+2s\hat{\beta}}
			|\varphi|^2dxdt\Bigr). 
   	 \end{array}
    \end{equation} 
    
	On the other hand, by considering a function $\eta\in C^1([0,T])$ such that $\eta\equiv 1$ in 
	$[0, T/2]$ and $\eta\equiv 0$ in $[3T/4,T]$, we can prove that $\eta\varphi$ satisfies the system
	\begin{equation}\label{carleman2.aux2}
    \left\{
    \begin{array}{lll}
        -(\eta\varphi)_t-\eta\varphi_{xxx}-\nu(t)\eta\varphi_{xx}-\overline{y}\eta\varphi_x=-\eta g-\eta'\varphi 
        &\mbox{in} &Q,\\
        (\eta\varphi)(0,t)=(\eta\varphi)(L,t)=0  &\mbox{on}& (0,T),\\
        (\eta\varphi)_x(0,t)=(\eta\varphi)_x(L,t)  &\mbox{on}& (0,T),\\
        (\eta\varphi)(\cdot,T)=0 &\mbox{in}& (0,L).        
    \end{array}\right.
	\end{equation}
	Additionally, from classical energy estimates and 
	regularity result with right-hand side in $L^2(Q)$ (see \eqref{ine.regularity.linearized.system}),
	we get
	\begin{equation*}
		\|\varphi(0)\|^2_{L^2(0,L)}+\|\varphi\|^2_{L^2(0,T/2;L^2(0,L))}
		\leq C\Bigl(\|g\|^2_{L^2(0,3T/4;L^2(0,L))}+\|\varphi\|^2_{L^2(T/2,3T/4;L^2((0,L))}\Bigr). 
	\end{equation*}
	Taking into account that $$\tau^5e^{-2s\hat\beta}\geq C>0,\,\,\,\forall t\in[T/2,3T/4]
	\quad\mbox{and}\quad e^{-4s\hat{\beta}}\geq C>0, \,\,\,\forall t\in[0,3T/4],$$
	we have
	\begin{equation}\label{carleman2.aux3}
	    \begin{array}{lll}
	    &\|\varphi(0)\|^2_{L^2(0,L)}
		+\displaystyle\int\limits_{0}^{T/2}\displaystyle\int\limits_0^L[\tau^5|\varphi|^2
		+\tau^3|\varphi_x|^2+\tau|\varphi_{xx}|^2]e^{-4s\hat\beta}dxdt\\
		&\leq C\Biggl(\displaystyle\int\limits_{0}^{3T/4}\displaystyle\int\limits_0^L|g|^2e^{-2s\hat{\beta}}
		dxdt
		+\displaystyle\int\limits_{T/2}^{3T/4}\displaystyle\int\limits_0^L \tau^5e^{-4s\hat\beta}|\varphi|^2dxdt
		\Biggr).
	 \end{array}
	\end{equation}
	Putting together \eqref{carleman2.aux1} and \eqref{carleman2.aux3} we obtain the desired inequality 
	\eqref{nullcontrol.ine.carleman}.\end{proof}
	Now, we can prove the null controllability of system \eqref{controllinear.mainsyst}. The idea is to look a
	solution $y$ in a suitable weight functional space. To this end, we introduce the following space:
	\begin{equation*}
    \begin{aligned}
    	E:=&\{(y,v): e^{s\hat\beta}y\in L^2(Q), 
    	\tau^{-9/2}e^{3s\breve{\beta}-s\hat{\beta}}v1_{\omega}\in L^2(Q),\\
    	& e^{s\hat{\beta}}\tau^{-3/2}y\in C([0,T];L^2(0,L))\cap L^2(0,T;H^1(0,L)),\\
    	&\hspace{0.8cm}e^{2s\hat\beta}\tau^{-5/2}(y_t+y_{xxx}-\nu(t)y_{xx}+\overline{y}y_x+y\overline{y}_x
    	-v1_{\omega})\in L^2(0,T;H^{-1}(0,L))\}.
    \end{aligned}
	\end{equation*}
	 
	\begin{proposition}\label{proof.prop.controllinear} 
	Consider $y_0\in L^2(0,L)$ and $e^{2s\hat{\beta}}\tau^{-5/2}h\in L^2(Q)$. Then,
	there exists a function $v\in L^2(0,T;L^2(\omega))$ such that the associated solution $(y,v)$ to
	\eqref{controllinear.mainsyst} satisfies $(y,v)\in E$.
	
	\noindent Furthermore, there exists a positive constant $C$ such that
	
	\begin{equation}\label{ine.prop.nullcontrol.linearcase}
		\|v\|_{L^2(0,T;L^2(\omega))}\leq C(\|y_0\|_{L^2(0,L)}+\|h\|_{L^2(Q)}).	
	\end{equation}
 
	\end{proposition}
	\begin{proof} The proof follows some ideas  
		\cite{guerrero2018local} and therefore we only give a sketch of the proof.  Let us now set
		\begin{equation*} 
		P_0=\{\varphi\in C^3(\overline{Q}): \varphi(0,t)= \varphi(L,t)=0,\,\,  
		\varphi_x(0,t)= \varphi_x(L,t),\,\,\mbox{on}\,\,(0,T)\}	
		\end{equation*}
		as well as the bilinear form
		\begin{equation*}
		a(\hat\varphi,w):=\iint\limits_{Q}e^{-2s\hat\beta}(L^*\hat\varphi)(L^*w)dxdt
		+\iint\limits_{\omega\times(0,T)}e^{-6s\breve{\beta}+2s\hat{\beta}}\tau^9\hat\varphi wdxdt,
		\quad \forall w\in P_0	
		\end{equation*}
 		and the linear form
 		\begin{equation}\label{sect.control.linearform}
 		\langle G,w\rangle:=\iint\limits_{Q}hwdxdt+\int\limits_{0}^L y_{0}(\cdot)w(\cdot,0)dx,	
 		\end{equation}
		where $L^*$ is the adjoint operator of $L$, i.e., 
		$$L^*w=-w_t-w_{xxx}-aw_{xx}-\overline{w}w_x.$$
		Note that Carleman inequality \eqref{nullcontrol.ine.carleman} holds for every $w\in P_0$, so that 
		we have 
		
		\begin{equation*}
			\iint\limits_{Q}\tau^5e^{-4s\hat\beta}|w|^2dxdt\leq Ca(w,w),\quad \forall w\in P_0.	
		\end{equation*}
		In consequence, it is very easy to prove that $a(\cdot,\cdot):P_0\times P_0\to\mathbb{R}$ is a
		symmetric, definite positive bilinear form on $P_0$, so that, by defining $P$ as the completion of 
		$P_0$ for the form induced by $a(\cdot,\cdot)$, it implies that $a(\cdot,\cdot)$ is well--defined,
		continuous and again definite positive on $P$. In addition, from Carleman inequality 
		\eqref{nullcontrol.ine.carleman} and the hypothesis over the function $h$, i.e., 
		$e^{2s\hat\beta}\tau^{-5/2}h\in L^2(Q)$, the linear form $w\rightarrow \langle G, w\rangle$ is 
		well defined and continuous on $P$. Hence, Lax--Milgram's lemma allows us to guarantee the existence
		and uniqueness of $\hat\varphi\in P$ satisfying
		
		\begin{equation}\label{existence.uniqueness.laxmilgram}
			a(\hat\varphi,w)=\langle G, w\rangle\,\,;\,\forall w\in P.	
		\end{equation}
		Let us set
		
		\begin{equation}\label{controlandsolution}\left\{
		\begin{array}{lll}
			\hat{y} &:=e^{-2s\hat\beta}L^*\hat\varphi\quad &\mbox{in}\,\, Q,\\
			\hat{v} &:=-e^{-6s\breve{\beta}+2s\hat{\beta}}\tau^9\hat\varphi\quad &\mbox{in}
			\,\, \omega\times(0,T),
		\end{array}
		\right.
		\end{equation}
		Observe that $\hat{y}$ verifies
		
		\begin{equation}\label{eq.bilinear.aux.for.y.v}
			a(\hat{\varphi},\hat{\varphi})=\iint\limits_{Q}e^{2s\hat{\beta}}|\hat{y}|^2dxdt
			+\iint\limits_{\omega\times(0,T)}e^{6s\breve{\beta}-2s\hat{\beta}}\tau^{-9}|\hat{v}|^2dxdt<+\infty.
		\end{equation}
		
		On the other hand, if $v$ is replaced by $\hat{v}$ in \eqref{controllinear.mainsyst}, we can introduce 
		$\tilde{y}$ as the weak solution
		of \eqref{controllinear.mainsyst}. It implies that $\tilde{y}$ is the unique solution of 
		\eqref{controllinear.mainsyst} with $v=\hat{v}$ defined by transposition (see Definition \ref{def1}). Then
		$\tilde{y}=\hat{y}$ is the weak solution to \eqref{controllinear.mainsyst}.
		
		\noindent Finally, we must verify that $(\hat{y},\hat{v})\in E$. Clarify, from \eqref{eq.bilinear.aux.for.y.v} 
		we know that $e^{s\hat\beta}\hat{y}\in L^2(Q)$ and $\tau^{-9/2}e^{3s\breve{\beta}-s\hat{\beta}}\hat{v}\in L^2(Q)$.
		Moreover, the second hypothesis of Proposition \ref{proof.prop.controllinear} guarantees that 
		$$e^{2s\hat\beta}\tau^{-5/2}(\hat{y}_t+\hat{y}_{xxx}-\nu(t)\hat{y}_{xx}+\overline{y}\hat{y}_x
		+\hat{y}\overline{y}_x-\hat{v})\in L^2(Q).$$
		Thus,  we must just check that $e^{s\hat{\beta}}\tau^{-3/2}\hat{y}\in C([0,T];L^2(0,L))\cap L^2(0,T;H^1(0,L))$. 
		To do this, we define the functions
		$$ y^*:=e^{s\hat\beta}\tau^{-3/2}\hat{y}\quad\mbox{and}\quad h^*:=e^{s\hat\beta}\tau^{-3/2}(h+\hat{v}).$$
		Observe that $y^*$ satisfies the system
		
		\begin{equation*}
    	\left\{
    	\begin{array}{lll}
    	y^*_t+y^*_{xxx}-\nu(t)y^*_{xx}+\overline{y}y^*_x+y^*\overline{y}_x=h^*+(e^{s\hat\beta}\tau^{-3/2})_{t}\hat{y} 
    	&\mbox{in} &Q,\\
        y^*(0,t)=y^*(L,t)=0  &\mbox{on}& (0,T),\\
        y^*_x(0,t)=y^*_x(L,t)  &\mbox{on}& (0,T),\\
        y^*(\cdot,0)=e^{s\hat\beta(0)}\tau^{-3/2}(0)\hat{y}_0(\cdot) &\mbox{in}& (0,L),        
    	\end{array}\right.
    	\end{equation*}
    	Since $e^{s\hat\beta}h\in L^2(Q)$ and $2\hat\beta<3\breve{\beta}$ (see eq. \eqref{carl.def.weights}), we obtain 
    	that $h^*+(e^{s\hat\beta}\tau^{-3/2})_{t}\hat{y}\in L^2(Q)$, in particular in $L^2(0,T;H^{-1}(0,L))$. 
    	Furthermore, for  $\hat{y}_0\in L^2(0,L)$, Proposition \ref{wp.regularity-linearized-system} allows us to have 
    	$y^*\in C([0,T];L^2(0,L))\cap L^2(0,T;H^1(0,L))$.
    	
    	By considering $\hat{v}$ defined in \eqref{controlandsolution},  
    	the bilinear form \eqref{sect.control.linearform} and the identity \eqref{existence.uniqueness.laxmilgram}, we 
    	can deduce \eqref{ine.prop.nullcontrol.linearcase}. This concludes the sketch of the proof of Proposition 
    	\ref{proof.prop.controllinear}.
	\end{proof}
%%%%%%%%%%%%%%%%%%%%%%%%%%%%%%%%%%%%%%%%%%%%%%%%%%%
%%%%%%%%% LOCAL  EXACT CONTROLLABILITY %%%%%%%%%%%%%%%
%%%%%%%%%%%%%%%%%%%%%%%%%%%%%%%%%%%%%%%%%%%%%%%%%%%	
	\section{\normalsize{Local exact controllability to trajectories}}\label{section.local.exact.trajetories}

	\quad
	In this section we give the proof of Theorem \ref{main_theorem2} through fixed point
	arguments. In order to apply the results obtained in the previous sections we consider the following
	change of variable.
	Let us set $y-\overline{y}=:z$ and use this equality in \eqref{intro.control.system}, 
	where $\overline{y}$ solves \eqref{intro.trajectory.system}. It is easy to verify that $z$ satisfies
	\begin{equation}\label{syst.final.section.z}
    \left\{
    \begin{array}{lll}
        z_t+z_{xxx}-\nu(t)z_{xx}+(z\overline{y})_{x}+zz_x=v1_{\omega} 
        &\mbox{in} &Q,\\
        z(0,t)=z(L,t)=0  &\mbox{on}& (0,T),\\
        z_x(0,t)=z_x(L,t)  &\mbox{on}& (0,T),\\
        z(\cdot,0)=y_0-\overline{y}_0 &\mbox{in}& (0,L).        
    \end{array}\right.
	\end{equation}
	observe that this changes reduce our problem to a local null controllability for the solution $z$ of the 
	nonlinear problem \eqref{syst.final.section.z},i.e., we are looking a function control $v$ such that $z$
	solution of \eqref{syst.final.section.z} satisfies 
	\begin{equation}\label{condition.on.z.final}
		z(\cdot,T)=0\quad\mbox{in}\quad (0,L). 
	\end{equation}
	To do this, we will use the following inverse mapping theorem (see \cite{alekseev1987optimal}).
 	\begin{teo}\label{teo.inverse.mapping}
		Suppose that $\mathcal{B}_1,\mathcal{B}_2$ are Banach spaces and 
		$\mathcal{A}:\mathcal{B}_1\to \mathcal{B}_2$ is a continuously differentiable map. We assume that for
		 $b_1^0\in \mathcal{B}_1, b_2^0\in \mathcal{B}_2$ the equality
		\begin{equation}\label{b_1^0}
			\mathcal{A}(b_1^0)=b_2^0
		\end{equation}
		holds and $\mathcal{A}'(b_1^0):\mathcal{B}_1\to \mathcal{B}_2$ is an epimorphism. Then there exists 
		$\delta >0$ such that for any $b_2\in \mathcal{B}_2$ which satisfies the condition 
		$\|b_2^0-b_2\|_{\mathcal{B}_2}<\delta$ there exists a solution $b_1\in \mathcal{B}_1$ of the equation
		\begin{equation*}
			\mathcal{A}(b_1)=b_2.
		\end{equation*}
	\end{teo}
	
	In our framework, we use the above theorem with the spaces 
	$$\mathcal{B}_1:=E\quad\mbox{and}\quad 
	\mathcal{B}_2:=L^2(e^{2s\hat\beta}\tau^{-5/2}(0,T);L^2(0,L))\times L^2(0,L)$$ 
	and the operator $\mathcal{A}:\mathcal{B}_1\rightarrow\mathcal{B}_2$ defined by
	$\mathcal{A}(z,v):=(z_t+z_{xxx}-\nu(t)z_{xx}+(z\overline{y})_{x}+zz_x-v1_{\omega},z(0))$,
	for all $(z,v)\in E$.
	
	\noindent In order to apply Theorem \ref{teo.inverse.mapping}, it is necessary to prove that $\mathcal{A}$
	is of class $C^1(\mathcal{B}_1,\mathcal{B}_2)$. We start by assuming that 
	$\overline{y}\in C([0,T];L^2(0,L))\cap L^2(0,T;H^1(0,L))$. Observe that all terms in the definition of $\mathcal{A}$
	are linear (and consequently $C^1$), except for $zz_{x}$. Thus, we will prove that the bilinear operator
	$((z^1,v^1),(z^2,v^2))\rightarrow \frac{1}{2}(z^1z^2)_x$ is continuous from $E\times E$ to 
	$L^2(e^{2s\hat\beta}\tau^{-5/2}(0,T);L^2(0,L))$. In fact, notice that 
	$$e^{s\hat\beta}\tau^{-3/2}z\in C([0,T];L^2(0,L))\cap L^2(0,T;H^1(0,L)), \quad\quad  \forall (z,v)\in E.$$ 
	
	\noindent Then, we have
	\begin{equation*}
	    \begin{array}{lll}
	    \|e^{2s\hat\beta}\tau^{-5/2}(z^1z^2)_x\|_{L^2(Q)}&\leq &C\displaystyle\int\limits_{0}^Te^{2s\hat\beta}\tau^{-3}\|
	    z^1(\cdot,t)\|^2_{L^\infty(0,L)}e^{2s\hat\beta}\tau^{-3}\|z^2(\cdot,t)\|^2_{H^1(0,L)}\\
	    & &\hspace{1cm} +e^{2s\hat\beta}\tau^{-3}\|z^2(\cdot,t)\|^2_{L^\infty(0,L)}e^{2s\hat\beta}\tau^{-3}\|z^1(\cdot,t)\|
	    ^2_{H^1(0,L)} dt\\
	    &\leq &C\|z^1\|_{\mathcal{B}_1}\|z^2\|_{\mathcal{B}_1}.
		\end{array}
	\end{equation*}
	Now, observe that $\mathcal{A}'(0,0):\mathcal{B}_1\rightarrow\mathcal{B}_2$ is given by 
	$$\mathcal{A}'(0,0)(z,v)=(z_t+z_{xxx}-az_{xx}+(z\overline{y})_{x}-v1_{\omega},z(0))),\quad \forall 
	(z,v)\in \mathcal{B}_1.$$
	However,  the null controllability result proved in Proposition 
	\ref{proof.prop.controllinear} allows to deduce that the previous  functional is surjective.
	
	Therefore, an application of Theorem \ref{teo.inverse.mapping} gives the existence of a positive
	number $\delta$ such that, if $\|z(0)\|_{L^2(0,L)}\leq\delta$, we can find a control $v$ and an associated
	solution $z$ to \eqref{syst.final.section.z} satisfying \eqref{condition.on.z.final}.
	This finishes the proof of Theorem  \ref{main_theorem2}.


\section{Acknowledgments}
This work has been partially supported by Fondecyt grant 1180528 (Eduardo Cerpa), Fondecyt grant 3180100 (Cristhian Montoya), Conicyt--Basal Project FB0008 AC3E (Eduardo Cerpa), Simons Foundation 201615 and NSFC 11571244 (Bingyu Zhang).



%\begin{thebibliography}{00}


\end{document}